\documentclass[11pt]{amsart}
\usepackage{amsmath}
\usepackage{amssymb}
\usepackage{amsfonts}
\usepackage{epsfig}

\usepackage[all]{xy}
\newcommand{\s}{\sigma}

\newcommand{\al}{\alpha}

\newcommand{\QQ}{\mathbb{Q}}

\newcommand{\C}{\mathbb{C}}
\newcommand{\CC}{\mathbb{C}}
\newcommand{\ZZ}{\mathbb{Z}}
\newcommand{\PP}{\mathbb{P}}
\newcommand{\FF}{\mathbb{F}}
\newcommand{\RR}{\mathbb{R}}
\newcommand{\y}{\mathfrak{y}}

\newcommand{\Z}{\mathbb{Z}}

\newcommand{\A}{\mathcal{A}}

\newcommand{\LL}{\mathcal{L}}

\newtheorem{thm}{Theorem}[section]
\newtheorem*{thm*}{Theorem}

\newtheorem{conjecture}[thm]{Conjecture}
\newtheorem*{lemma*}{Lemma}
\newtheorem{definition}[thm]{Definition}
\newtheorem{example}[thm]{Example}
\newtheorem{prs}[thm]{Proposition}
\newtheorem{remark}[thm]{Remark}

\newtheorem{notation}[thm]{Notation}

\newtheorem*{notation*}{Notation}

\begin{document}
\title [Conjugation-Free groups, LCS and line arrangements]{Conjugation-free groups,  lower central series and line arrangements}

\author{Michael Friedman}

\address{Michael Friedman, Institut Fourier, 100 rue des maths, BP 74, 38402 St Martin d'H\'eres cedex, France}
\email{Michael.Friedman@ujf-grenoble.frþ}

\begin{abstract}
The quotients $G_k/G_{k+1}$ of the lower central series of a finitely presented group $G$ are an important invariant of this group.
In this work we investigate the ranks of these quotients in the case of a certain class of conjugation-free groups, which are groups generated by $x_1,\ldots,x_n$ and  having only cyclic relations:
 $$ x_{i_t} x_{i_{t-1}} \cdot \ldots \cdot x_{i_1} = x_{i_{t-1}} \cdot \ldots \cdot x_{i_1} x_{i_t} = \cdots = x_{i_1} x_{i_t} \cdot \ldots \cdot x_{i_2}.$$

 Using tools from group theory and from the theory of line arrangements we  explicitly find these ranks, which depend only at the number and length of these cyclic relations. It follows that for these groups  the associated graded Lie algebra $gr(G)$ decomposes, in any degree, as a direct product of local components.
\end{abstract}

\maketitle


\section{Introduction}

Let $G$ be a group and consider its lower central series defined by $G_1 = G, G_k = [G_{k-1},G], k \geq 2$. The graded sum
$$
gr(G) \doteq \bigoplus_{k \geq 1} G_k/G_{k+1}
$$
has a graded Lie algebra structure induced by the commutator bracket on $G$. In general, $gr(G)$ reflects many properties of the group $G$. For example, if $G$ is finitely generated, then the abelian groups $G_k/G_{k+1}$ are also finitely generated. It is, however, difficult  to find out the structure of $G_k/G_{k+1}$ even in the case when $G$ is finitely generated. Thus, determining this structure, or even the ranks of the quotients $\phi_k(G) \doteq \text{rank}(G_k/G_{k+1})$, for a certain class
of finitely generated groups is a significant problem.

In the case of $G = \FF_n$, the free group of rank $n$, Magnus (see e.g. \cite{MKS}) showed that $gr(\FF_n)$ is the free Lie algebra on $n$ generators, whose ranks were computed by Witt \cite{W}. Hall \cite{H} introduced the basic commutators of $\FF_n$, showing that the coset classes of weight $k$ form a basis of $G_k/G_{k+1}$. Therefore, for the free group, the structure of $gr(\FF_n)$ is completely known. Another well-known example is the  work of Chen
\cite{Ch}, computing the structure of $gr(\FF_n/(\FF_n)'')$.

In this paper we concentrate on a particular class of finitely presented groups: groups  having a presentation  where all the relations are cyclic relations; that is, of the form:
$$ x_{i_t}^{s_{t}} x_{i_{t-1}}^{s_{t-1}} \cdot \ldots \cdot x_{i_1}^{s_{1}} = x_{i_{t-1}}^{s_{t-1}} \cdot \ldots \cdot x_{i_1}^{s_{1}} x_{i_t}^{s_{t}} = \cdots = x_{i_1}^{s_{1}} x_{i_t}^{s_{t}} \cdot \ldots \cdot x_{i_2}^{s_{2}},$$
where $\{x_1,\ldots,x_n\}$ are the generators of the group, $2 \leq t$, $\{ i_1,i_2, \dots , i_t \} \subseteq \{1, \dots, n \}$ is an increasing subsequence of indices and $s_{j} \in \langle x_1,\ldots,x_n \rangle$. We call such groups \emph{cyclic-related} groups (see Definition \ref{defCyclicRelated}) and the above relation a \emph{cyclic relation} of length $t$. When all the conjugating elements are equal to the identity, i.e. $s_j = e$, then we call  such a group \emph{conjugation-free} (see Definition \ref{defConjFree}). The simplest example of such a group is the group $\FF_n/\langle R_n \rangle$ generated by $n$ generators with one cyclic relation of length $n$:
$$
R_n : \,\, x_n x_{n-1} \cdot \ldots \cdot x_1 =  x_{n-1}x_{n-2} \cdot \ldots \cdot x_1 x_n = \cdots = x_1x_n \cdot \ldots \cdot x_2.
$$
Obviously, $\FF_n/\langle R_n \rangle \simeq \FF_{n-1} \oplus \ZZ$, and thus  the structure of $gr(\FF_n/\langle R_n \rangle)$ is known.

 A less-known example of a cyclic-related group is the pure braid group $PB_n$ (see \cite[Theorem 2.3]{MM} for a cyclic-related presentation of this group). 
  Note that given an arrangement of hyperplanes $\A = \{H_1,\ldots,H_n\} \subset \CC^m$, the fundamental group $\pi_1(\CC^m - \A)$ is a cyclic-related group (see Remark \ref{RemRelCyclic}); thus, from this perspective, it is clear why the pure braid group is a cyclic-related group, as it is the fundamental group of the complement of a hyperplane arrangement known as the braid arrangement. In general, if $W$ is a real reflection group, $\A_W \subset \CC^m$ the associated hyperplane arrangement and $B_W$ the Artin group associated to $W$, then $\pi_1(\CC^m - \A_W) \simeq \text{ker}(B_W \twoheadrightarrow W)$
 is a cyclic-related group. Note that the study of the lower central series quotients of the pure braid group was initiated by Kohno \cite{K2}.

  As an arrangement of lines in $\CC^2$ is an example of a hyperplane arrangement, for any line arrangement $\LL$ the associated fundamental group $G=\pi_1(\CC^2 - \LL)$ is a cyclic-related group. Kohno \cite{K} proved, using the mixed Hodge structure on $H^1(\CC^2-\LL,\QQ)$,  that $gr(G) \otimes \QQ$ is isomorphic to the nilpotent completion of the holonomy Lie algebra of
 $\CC^2-\LL$. Falk proved \cite{Falk}, using Sullivan's 1-minimal models, that the lower central series ranks $\phi_k(G)$ are determined only by the combinatorics of $\LL$. However, a precise
combinatorial formula for $\phi_k(G)$, and even for $\phi_3(G)$, is not known (for the general formula for $\phi_3$, see e.g. \cite[Corollary 3.6]{SS} or Remark \ref{remPhi3For} below).

 These motivations lead us to investigate these ranks for cyclic-related and conjugation-free groups. While it is fairly easy to find out the rank $\phi_2$ for any cyclic-related group (see Section \ref{subsubsecG2G3}), it is a harder task when considering the rank $\phi_3$. We find that for a certain class of conjugation-free groups, there is an upper bound on $\phi_3$ (see Proposition \ref{prsFirstPart}). Explicitly, for such a group $G$, using only group-theoretic arguments, we prove that $\phi_3(G) \leq \sum_{i \geq 3} n_i\phi_3(\FF_{i-1})$, where $n_i$ is the number of cyclic relations of length $i$ of $G$.

  We then associate to this group $G$ a line arrangement $\LL(G)$, such that
 $\pi_1(\CC^2 - \LL(G))\simeq G$. By \cite{Falk2}, given any line arrangement $\LL$, $\phi_3(\pi_1(\CC^2 - \LL))$ is bounded from below by the above upper bound. Thus, for this class of conjugation-free groups, one can calculate $\phi_3$ directly, i.e. the
 third quotient of the lower central series behaves as if the  group is a direct product of free groups. From this it
follows, by \cite{PS}, that these conjugation-free groups are \emph{decomposable} (see Definition \ref{defDecomp}), that is, the quotient $G_k/G_{k+1}$  decomposes, for $k \geq 2$, as a direct product of local components and that $\phi_k(G) = \sum_{i \geq 3} n_i\phi_k(\FF_{i-1})$ for every $k \geq 2$ (see Theorem \ref{thmMainPhi}).
Hence, for this class of groups, we have found the complete structure of $gr(G)$.

\medskip

The paper is organized as follows. In Section \ref{secPrem} we define the main object of our research: cyclic-related and conjugation-free groups.
We prove in Section \ref{secLCSCGF} that for a certain class of conjugation-free groups, there is an upper bound on the third rank $\phi_3$. Associating to such a group a  line arrangement, we show in Section \ref{sectionLineArrCFgroup} that this upper bound is obtained. This leads us to find an explicit formula on the ranks $\phi_k$ for $k \geq 3$ and to find a new series of examples of decomposable groups.

{\textbf{Acknowledgements}}: The author would like to thank Michael Falk, Uzi Vishne, Roland Bacher and Mikhail Zaidenberg for stimulating talks, and the Fourier Institut in Grenoble for the warm hospitality and support.

\section{Cyclic-related and Conjugation-free groups} \label{secPrem}

In this section, we define the main object of our research: cyclic-related groups and conjugation-free groups. Let $G$ be a group, and denote $[a,b] \doteq a^{-1}b^{-1}ab$, $a^b = b^{-1}ab$ for $a,b \in G$.

\begin{definition} \label{defCyclicRelated} \rm{
Let $G$ be a finitely presented group. We say that $G$ is \emph{cyclic-related }
if $G$ has a presentation such that it is generated by $x_1,\ldots,x_n$ with relations $R_1,\ldots,R_w$ and the following requirements hold:
\begin{enumerate}
\item All the relations are of the following form:
$$R_p = R_{p,(i_t,  \ldots ,i_1)}: \,\,\, x_{i_t}^{s_{p,t}} x_{i_{t-1}}^{s_{p,t-1}} \cdot \ldots \cdot x_{i_1}^{s_{p,1}} = x_{i_{t-1}}^{s_{p,t-1}} \cdot \ldots \cdot x_{i_1}^{s_{p,1}} x_{i_t}^{s_{p,t}} = \cdots = x_{i_1}^{s_{p,1}} x_{i_t}^{s_{p,t}} \cdot \ldots \cdot x_{i_2}^{s_{p,2}},$$
where $ 1 \leq p \leq w$, $2 \leq t$, $\{ i_1,i_2, \dots , i_t \} \subseteq \{1, \dots, n \}$ is an increasing subsequence of indices, $s_{p,j} \in \langle x_1,\ldots,x_n \rangle$ for $1 \leq j \leq t$. These relations are called \emph{cyclic relations of length} $t$. Note that when $t=2$ we get the commutator $[x_{i_1}^{s},x_{i_2}^{s'}]=e$.
\item For every pair of indices $j_1,j_2 \in \{1, \dots, n\}$, $j_1 \neq j_2$, there is a unique relation $R_{p,(i_t,\ldots,i_1)}$ such that
$\{ j_1,j_2\} \subseteq \{ i_1,i_2, \dots , i_t \}$.
\item For every two relations $R_{p,(i_t,\ldots,i_1)}$, $R_{p',(j_s,\ldots,j_1)}$, $p \neq p'$ we have that: $$|\{ i_1,i_2, \dots , i_t \} \cap \{ j_1,j_2, \dots , j_s \}| \leq 1.$$
\end{enumerate}
}
\end{definition}

\begin{example}
\emph{(1) The pure braid group $PB_n$ is an example of a cyclic-related group; see \cite[Theorem 2.3]{MM}  for a presentation of $PB_n$ with generators and relations satisfying the requirements above.\\ (2) For any hyperplane arrangement $\A \subset \CC^k$,
$\pi_1( \CC^k - \A)$ is a cyclic-related group (see Remark \ref{RemRelCyclic}(1)).}
\end{example}

\begin{remark} \label{remListComm} \rm{
Note that  $R_{p,(i_t, \ldots ,i_1)}$ a cyclic relation of length $t$ can be written as a list of $t-1$ commutators:
$$
[x_{i_k}^{s_{p,k}}, x_{i_{k-1}}^{s_{p,k-1}} \cdot \ldots \cdot x_{i_1}^{s_{p,1}} \cdot  x_{i_t}^{s_{p,t}} \cdot \ldots \cdot x_{i_{k+1}}^{s_{p,k+1}}] = e,
$$
where $1 \leq k \leq t$.}
\end{remark}

\begin{definition} \label{defMultRel} \rm{
We say that a cyclic relation of length $t$ is \emph{multiple} if $t \geq 3$.
}
\end{definition}

\begin{definition} \label{defGraphCyclicRelated} \rm{
Given a cyclic-related group $G$, we define its \emph{associated graph} $Gr(G)$ in the following way:
 \begin{itemize}
\item Vertices: for every multiple cyclic relation $R_p = R_{p,(i_t,\ldots,i_1)}$ (where $t \geq 3$) associate a vertex $v_p = v_{p,(i_t,\ldots,i_1)}$. When it is clear what are the corresponding generators to $v_{p,(i_t,\ldots,i_1)}$ we write only $v_p$.
\item Edges: two vertices $v_p = v_{p,(i_t,\ldots,i_1)}$, $v_{p'} = v_{p',(j_s,\ldots,j_1)}$ are connected by a edge $e_{i_v}$ if
$$\{ i_1,i_2, \dots , i_t \} \cap \{ j_1,j_2, \dots , j_s \} = \{i_v\}  = \{j_u\},$$
where $1 \leq v \leq t, 1 \leq u \leq s$.
\end{itemize}
}
\end{definition}

\begin{definition} \label{defConjFree} \rm{
Let $G$ be  a cyclic-related group. If  every cyclic relation is of the form
$$R_{p,(i_t,\ldots,i_1)}: \,\,\, x_{i_t} x_{i_{t-1}} \cdot \ldots \cdot x_{i_1} = x_{i_{t-1}} \cdot \ldots \cdot x_{i_1} x_{i_t} = \cdots = x_{i_1} x_{i_t} \cdot \ldots \cdot x_{i_2},$$
(that is, there are no conjugations on the generators appearing in the relation) then we say that $G$ is \emph{conjugation-free} and that its relations are \emph{conjugation-free}.
}

\begin{example} \label{exmGraphCycRel} \rm{ (1) If the graph $Gr(G)$ of a conjugation-free group $G$ is empty, then all the relations of $G$ are commutators between all the generators and hence $G$ is a free abelian group, i.e. if $G$ is generated by $n$ generators then $G \simeq \ZZ^n$. This is since, by requirement (2) of Definition \ref{defCyclicRelated}, every generator commutes with all other generators. Note that this is not true any more if $G$ is cyclic-related, but not conjugation-free. For example, for the cyclic-related group
$$
G_1 \doteq \langle x_1, x_2, x_3 : [x_1,x_2^{x_3}] = [x_1,x_3^{x_2}] = [x_2,x_3^{x_1}] = e \rangle,
$$
the graph $Gr(G_1)$ is empty (as there are no multiple relations) but $G_1$ is not abelian (one can check that there exists an epimorphism to $Sym_3$).

(2) Let $G$ be a conjugation-free group with $n$ generators and one cyclic relation of length $n$, i.e.
$$
G \simeq \langle x_1,\ldots,x_n : x_nx_{n-1} \cdot \ldots \cdot x_1 = \cdots = x_1x_n \cdot \ldots \cdot x_2  \rangle.
$$
Then $G \simeq \FF_{n-1} \oplus \Z$ and the graph consists of only one vertex.

(3)  Let $G$ be a conjugation-free group such that $Gr(G)$ is a disjoint union of two graphs. Then  $G$ is a direct sum of two conjugation-free groups, whose graphs correspond to the two graphs which are the components of $Gr(G)$.

(4) If the graph of a conjugation-free group consists of only vertices (i.e. there are no edges), then, by the previous examples, this group is isomorphic to a direct sum of a free abelian group and free groups. See remark \ref{remCFnoCyc} for a generalization of this when $Gr(G)$  does not have  cycles.

(5) Let $G_2$ be a cyclic-related group generated by $x_1,\ldots,x_6$ with the following relations:
 \begin{itemize}
\item cyclic relations of length $3$:
$$
R_1:x_3^{x_4} x_2 x_1 = x_2 x_1 x_3^{x_4} = x_1 x_3^{x_4} x_2, \,\, R_2:x_6 x_5 x_1 = x_5 x_1 x_6 = x_1 x_6 x_5,\,\,R_3: x_5 x_4 x_3 = x_4 x_3 x_5 = x_3 x_5 x_4.
$$
\item cyclic relations of length $2$:
$$
R_4:x_4 x_2 = x_2 x_4, \,\, R_5:x_5 x_2^{x_3} = x_2^{x_3} x_5,\,\, R_6:x_6 x_2 = x_2 x_6,\,\, R_7:x_6 x_4 = x_4 x_6
$$
\end{itemize}

Then, the vertices of $Gr(G_2)$ are $v_1,v_2,v_3$ (associated to the relations $R_1,R_2,R_3$) and the edges are $e_{1}, e_{3}$ and $e_{5}$ (see Figure \ref{ExampGr}). Note that with this presentation $G_2$ is not conjugation-free.

\begin{figure}[h!]
\epsfysize 3cm
\epsfbox{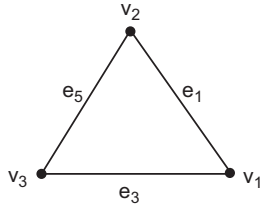}
\caption{The graph $Gr(G_2)$ associated to the cyclic-related group $G_2$.}
\label{ExampGr}
\end{figure}
}
\end{example}

\begin{remark}\emph{
Let $G$ be a cyclic-related group. By requirement (3) in Definition \ref{defCyclicRelated}, every two vertices of $Gr(G)$ are connected with at most one edge.}
\end{remark}
%

\end{definition}

%

\section{Lower central series and conjugation-free groups} \label{secLCSCGF}

In this section we prove the main result: that for a certain class of conjugation-free groups, the rank of  the third quotient of the lower central series is bounded from above. We will see in Section \ref{subsecLCSLineArr} that this inequality is in fact an equality, using  tools arising from line arrangements.

 First, we give in Section \ref{subsecLCS} a short survey  of some of the known results concerning the quotients of the lower central series, afterwards proving the main theorem in Section \ref{subsecG3G4}.

\subsection{The lower central series} \label{subsecLCS}
We begin with some notations. Let $G$ be a group generated by $x_1,\ldots,x_n$ and consider its lower central series $G_k$, where $G_1 = G$ and $G_k = [G_{k-1},G],\,k\geq 2$. Denote $\phi_k(G) = \text{rank}(G_k/G_{k+1})$. If $G = \FF_n$ the free group with $n$ generators, then we denote
$\omega_k(n)~\doteq~\phi_k(\FF_n) $.

\begin{notation}
\begin{enumerate}
\item $x_{i,j} \doteq [x_i,x_j] \in G_2$ for $i,j \in \{1,...,n\}$.
\item $x_{i,jk} \doteq [x_i,x_jx_k], x_{ij,k} \doteq [x_ix_j,x_k] \in G_2$ for $i,j,k \in \{1,...,n\}$.
\item $x_{i,j,k} \doteq  [x_{i,j},x_k] = [[x_i,x_j],x_k] \in G_3$ for $i,j,k \in \{1,...,n\}$.
\item $x_{i,j,kl} \doteq  [x_{i,j},x_kx_l] = [[x_i,x_j],x_kx_l] \in G_3$ for $i,j,k,l \in \{1,...,n\}$.
\item \emph{More generally, if $\al \in G$, then we denote
 $x_{i,\al} = [x_i,\al]$ and
 $x_{i,\al,k} = [x_{i,\al},x_k] \in G_3$ and in the same way for expressions of the form}
$x_{\al,j},\, x_{\al,j,k}$\emph{ or }$x_{i,j,\al}$.
\end{enumerate}
\end{notation}

\begin{remark}
\emph{If $G$ is a group generated by $x_1,\ldots,x_n$, then it is well-known that
 $G_2/G_3$ is generated by $x_{j,i}$ where $j>i$ and
 $G_3/G_4$ is generated by $x_{j,i,k}$ where $j>i$ and $k \geq i$ (see \cite{H}).}
\end{remark}
\noindent
Concentrating on the groups $G_2/G_3$ and $G_3/G_4$, we recall a few of their well-known properties.
\begin{prs} \label{prsG3G4}
\emph{(I)} For every $i,j,k \in \{1,\ldots,n\}$, the following  equivalences hold in $G_2/G_3$:
$$
[x_i,x_j^{x_k}] \equiv [x_i,x_j],\,\, x_{i,jk} \equiv x_{i,k}x_{i,j}.
$$
\noindent
\emph{(II)} The following  equivalences hold in $G_3/G_4$:

\noindent
  \begin{equation} \label{eqnEqG3G4}
  x_{i,j,kl} \equiv x_{i,j,l}x_{i,j,k},\,\, x_{i,jk,l} \equiv x_{i,k,l}x_{i,j,l}, \,\, x_{ij,k,l} \equiv x_{j,k,l}x_{i,k,l}.\end{equation}

\begin{equation} \label{eqnXijk}
 x_{i,j,k} \equiv x_{j,i,k}^{-1},\end{equation}
 where $i,j,k,l$ are either words in $G$ or indices in $\{1,\ldots,n\}$.
\end{prs}

\begin{proof}
(I)  Obvious. \\
(II)(1) We give the proof only for the expression $x_{i,jk,l}$. The proof for the other expressions is similar.\\
$$x_{i,jk,l} = [[x_i,x_jx_k],x_l] = [[x_i,x_k][x_i,x_j]^{x_k},x_l] = [[x_i,x_k][x_k,[x_i,x_j]^{-1}][x_i,x_j],x_l].$$ When expanding the commutator brackets of the right hand side, we see that both the expression $[x_k,[x_i,x_j]^{-1}]$ and its inverse appear. As $[x_l,[x_i,x_j]^{-1}] \in G_3$, it commutes in the group $G_3/G_4$  with any other element. Therefore,
$$x_{i,jk,l} \equiv [[x_i,x_k][x_i,x_j],x_l] \equiv [[x_i,x_k],x_l] \cdot [[x_i,x_j],x_l] (\text{mod }G_4)  ,$$ where in the last equivalence we used that $[ab,c] = [a,c]^b[b,c] = [b,[a,c]^{-1}][a,c][b,c]$, and that $[b,[a,c]^{-1}] \in G_4$ if $a \in G_2$.

(2) Let $a = x_{i,j}$; thus $a^{-1} = x_{j,i}$. Therefore $x_{i,j,k} = [a,x_k] = a^{-1}x_k^{-1}ax_k$ and $x_{j,i,k} = [a^{-1},x_k] = ax_k^{-1}a^{-1}x_k$. Hence, $a^{-1}x_{j,i,k}a = x_k^{-1}a^{-1}x_ka = x_{i,j,k}^{-1}$. Thus, in $G_3/G_4$, we get that $x_{i,j,k} \equiv x_{j,i,k}^{-1}.$
\end{proof}

\subsubsection{The rank of $G_2/G_3$} \label{subsubsecG2G3}

From now on, let $G$ be a cyclic-related group with relations $R_1,\ldots,R_w$. In this subsection we  give a combinatorial description of $\phi_2(G)$ (see also Remark \ref{RemPhi23Top}(1) for a description of $\phi_2(G)$ via topological invariants). Recall again that $G_2/G_3$
is generated by $x_{j,i}$ when $j>i$.

Note that by Remark \ref{remListComm} and Proposition \ref{prsG3G4}(I), every cyclic relation of length $t$ in $G$ is equivalent, in $G_2/G_3$, to a list of $t-1$ commutators, where the generators appear without conjugations. That is, the relation $R_p$ of $G$:
$$R_p = R_{p,(i_t,\ldots,i_1)}: \,\,\, x_{i_t}^{s_{p,t}} x_{i_{t-1}}^{s_{p,t-1}} \cdot  \ldots \cdot  x_{i_1}^{s_{p,1}} = x_{i_{t-1}}^{s_{p,t-1}} \cdot  \ldots \cdot  x_{i_1}^{s_{p,1}} x_{i_t}^{s_{p,t}} = \cdots = x_{i_1}^{s_{p,1}} x_{i_t}^{s_{p,t}}  \cdot  \ldots \cdot  x_{i_2}^{s_{p,2}},$$
is equivalent in $G_2/G_3$ to the following list of commutators:
$$
R_{p,k}:\,\,[x_{i_k}, x_{i_{k-1}}  \cdot  \ldots \cdot  x_{i_1} \cdot x_{i_t}  \cdot  \ldots \cdot  x_{i_{k+1}}] = e,
$$
where $1 \leq k \leq t$. Moreover, by Proposition \ref{prsG3G4}(I) we see that every relation $R_{p,k}$, where $1 \leq k \leq t$, is equivalent to
\begin{equation} \label{eqnRelG2G3}
x_{i_k,i_{k+1}} \cdot \ldots \cdot x_{i_k,i_{t}} \cdot x_{i_k,i_{1}} \cdot \ldots \cdot x_{i_k,i_{k-1}} = e.
\end{equation}

By requirement (2) of Definition \ref{defCyclicRelated}, for every pair of indices $i,j$, the generator $x_{j,i}$ of $G_2/G_3$ appears as a term in the relations of $G_2/G_3$ only once. Therefore, for every two different relations $R_p,R_{p'}$ of $G$, when considering these relations
in $G_2/G_3$ (as in Equation (\ref{eqnRelG2G3})),
the generators that appear in these relations are different.

For example, every cyclic relation of length $3$: $x_k^{\alpha} x_j^{\beta} x_i^{\gamma} = x_j^{\beta} x_i^{\gamma} x_k^{\alpha} = x_i^{\gamma} x_k^{\alpha} x_j^{\beta}$ (where $\alpha,\beta,\gamma \in G$ and $k > j > i$) is equivalent to  equalities
of the form $x_{j,i} \equiv x_{k,j} \equiv x_{k,i}^{-1}$ in $G_2/G_3$. Hence, while in every such cyclic relation appear three generators of $G_2/G_3$, two of them can be expressed as the third (or as its inverse).

In the same way, while in a cyclic relation $R_p$ (in $G_2/G_3$) of length $m$ appear $\binom{m}{2}$ generators of $G_2/G_3$, $m-1$ of them can be expressed
as a product of the others; thus a cyclic relation $R_p$ of length $m$ contributes $v(p) \doteq \binom{m}{2} - m +1$ independent generators to $G_2/G_3$. Note that $v(p) = \binom{m-1}{2} = \omega_2(m-1)$.

 Hence $ \sum_{p=1}^w v(p) = \sum_{i \geq 3} n_i\omega_2(i-1)$, where  $n_i$ is the number of cyclic relations of length $i$.
 We therefore see that  $$\text{rank}(G_2/G_3) = \phi_2(G) = \sum_{i \geq 3} n_i\omega_2(i-1).$$

\subsection{The group $G_3/G_4$} \label{subsecG3G4}
Here we would like to generalize the above equality, proved for the rank of $G_2/G_3$, for the group $G_3/G_4$, for a special class of conjugation-free groups. We first introduce some notations.

\begin{notation}
\emph{Let $G$ be a cyclic-related group.\\(1) Let $n_m$ be the number of cyclic relations of length $m$.\\ (2) Let $\beta(Gr(G))$ be the first Betti number of the graph $Gr(G)$.}
\end{notation}

Thus, we want to prove the following theorem:

\begin{thm} \label{ThmMain}
Let $G$ be a conjugation-free group such that $\beta(Gr(G)) \leq 1$. Then:
$$\phi_3(G) = \sum_{i \geq 3} n_i\omega_3(i-1).$$
\end{thm}

The proof of this theorem is divided into two parts. The first part, given in Proposition \ref{prsFirstPart}, proves that
$\phi_3(G) \leq \sum_{i \geq 3} n_i\omega_3(i-1)$, using only arguments from group-theory. The second part, given in section \ref{subsecLCSLineArr},
proves that
$\phi_3(G) \geq \sum_{i \geq 3} n_i\omega_3(i-1)$ and  uses arguments from the theory of line arrangements. In fact, as will be seen in Remark \ref{remRelCond}, the condition that $\beta(Gr(G)) \leq 1$ can be relaxed. Before proving these propositions, we need to examine closely the generators of $G_3/G_4$.

\begin{example} \label{exmNlines}
\emph{Let $G \simeq \langle x_1,\ldots,x_n : x_nx_{n-1} \cdot \ldots \cdot x_1 = \cdots = x_1x_n \cdot \ldots \cdot x_2  \rangle$.
As $G \cong \FF_{n-1} \oplus \ZZ$, we get that $\phi_k(G)~=~\omega_k(n-1)$.}
\end{example}

\begin{definition}\label{defLocalGen}
\emph{ Let $G$ be a conjugation-free group generated by $x_1,\ldots,x_n$, with the relations $R_1,\ldots, R_w$.\\ (1) A generator of $G_3/G_4$ of the form $x_{j,i,k} = [[x_j,x_i],x_k]$ is called \emph{local} if there is a multiple relation $R_p = R_{p,(u_t,\ldots,u_1)}$ ($2<t, 1 \leq p \leq w $) such that $\{i,j,k\} \subseteq \{u_r\}_{r=1}^t$.\\ (2) We say that a generator $x_k$ \emph{participates} in a cyclic relation $R_p = R_{p,(u_t,\ldots,u_1)}$ if $k \in \{u_r\}_{r=1}^t$.}
\end{definition}

\begin{remark}
\emph{Every conjugation-free cyclic relation of length $t$ induces $t^3-t$ local generators of $G_3/G_4$, which are not induced from other relations (since we subtract the number of generators of the form $x_{i,i,i}$. Note also that $x_{i,i,i}=e$). However,
by Example \ref{exmNlines}, from such relation there are only  $\omega_3(t-1)$ non-trivial independent local generators, that is, the other local generators can be expressed as a multiplication of powers of the other basic local generators.}
\end{remark}

We are now ready to prove the first part of Theorem \ref{ThmMain}:

\begin{prs} \label{prsFirstPart}
Let $G$ be a conjugation-free group such that $\beta(Gr(G)) \leq 1$. Then:
$$\phi_3(G) \leq \sum_{i \geq 3} n_i\omega_3(i-1).$$
\end{prs}

\begin{proof}
Assume that $G$ is generated by $n$  generators $x_1,\ldots,x_n$.
In order to prove that\linebreak $\phi_3(G) \leq \sum_{i \geq 3}n_i\omega_3(i-1)$, we  prove that every \emph{non}-local generator of the form $x_{j,i,k} = [[x_j,x_i],x_k]$ (where $j>i$ and $k \geq i$) is equivalent (modulo $G_4$) to the identity element $e$. From now on, we denote by $\equiv$ equalities that take place in $G_3$, modulo the group $G_4$.

\medskip
We prove this claim by splitting our proof into cases.
If the generators  $x_i,x_j$ commute, then $x_{j,i}=e$ and thus $x_{j,i,k} = e$. Assume thus that the generators $x_j,x_i$ participate in the same multiple relation $R_1$, and let $v_1$ be the corresponding vertex to this relation in $Gr(G)$. Explicitly, as $j>i$,
$$
R_1 = R_{1,(y''_{k''},\ldots, y''_0,j,y'_{k'},\ldots,y'_0,i,y_k,\ldots,y_0)}.
$$

\begin{remark} \label{remEmptySet}
\rm{Note that some (but not all) of the sets $\{y_0,\ldots,y_k\}, \{y'_0,\ldots,y'_{k'}\}, \{y''_0,\ldots,y''_{k''}\}$ may be empty.}
\end{remark}

 If $x_{i,k}=x_{j,k}=e$, that is, the generator $x_k$ commutes with the generators
$x_i$ and $x_j$, then $x_{j,i,k} = e$. We therefore examine the cases when the generator $x_k$ participates with at least one of the generators, $x_i$ or $x_j$, in the same multiple relation. We consider three cases: when $x_{i,k}=e$ but $x_{j,k} \neq e$; when $x_{j,k}=e$ but $x_{i,k} \neq e$; and when $x_{i,k} \neq e$ and $x_{j,k} \neq e$.

\medskip
\noindent

\textbf{First case}: In this case $x_k, x_i$ commute, i.e.  $x_{k,i}=e$, and $x_k, x_j$ participate in the same multiple relation $R_2$; let $v_2$ be the corresponding vertex to this relation in $Gr(G)$ (see Figure \ref{Case1n}).

\begin{figure}[h!]
\epsfysize 0.7cm
\epsfbox{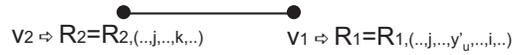}
\caption{Part of the graph $Gr(G)$: The generators $x_i, x_j$ participate in a multiple relation $R_1$ (to whom the vertex $v_1$ is associated) and the generators $x_k, x_j$ participate in a multiple relation $R_2$ (to whom the vertex $v_2$ is associated when we assume that $k>j$). The generator $x_{y'_u}$ is a another generator participating in $R_1$.}
\label{Case1n}
\end{figure}

\noindent
 Let us prove that  $x_{j,i,k}\equiv e$.
\begin{equation} \label{eqnIniEqn}
x_{j,i,k}^{-1} \overset{Eqn. (\ref{eqnXijk})}{\equiv} x_{i,j,k} = x_j^{-1}x_i^{-1}x_jx_ix_k^{-1}x_i^{-1}x_j^{-1}x_ix_jx_k \overset{x_{k,i}=e}{=}
x_j^{-1}x_i^{-1}x_jx_k^{-1}x_j^{-1}x_ix_jx_k.
\end{equation}
\noindent

Denote $Y = x_{y_k}\cdot \ldots \cdot x_{y_0}, Y' = x_{y'_{k'}}\cdot \ldots \cdot x_{y'_0}, Y'' = x_{y''_{k''}}\cdot \ldots \cdot x_{y''_0}$.
Therefore, since $G$ is conjugation-free, by Definition \ref{defConjFree} part of the relations of $R_1$ are the following:

\begin{equation} \label{relPointY}
Y''x_jY'x_iY = x_jY'x_iYY'' = Y'x_iYY''x_j = x_iYY''x_jY' = YY''x_jY'x_i.
\end{equation}

\begin{remark} \label{remConj}
\emph{Let us note that if $a \in G_3$ then for every $b \in G$, $a \equiv b^{-1}ab $ (mod($G_4$)), that is, we can conjugate an element in $G_3$ by any element in $G$ and remain in the same conjugacy class.}
\end{remark}

We now consider two cases: either that $[x_k,Y] = [x_k,Y'] = [x_k,Y''] = e$ or that at least one of these equalities does not hold.

 \medskip

\textbf{Case (1)}: Assume that $[x_k,Y] = [x_k,Y'] = [x_k,Y''] = e$. This happens when all the
 generators participating in $R_1$, except the generator $x_j$, commute with $x_k$ (for example, this situation occurs when the vertex $v_1$ is not a part of the cycle of the graph  $Gr(G)$).

By Remark \ref{remConj}, we can conjugate $x_{j,i,k}^{-1}$ by any element. Thus, following Equation (\ref{eqnIniEqn}):
\begin{equation} \label{eqnXjik_1}
x_{j,i,k}^{-1} \overset{Conj.\,\,by\,\,YY''Y'}{\equiv} (YY''Y')^{-1}\cdot x_j^{-1}x_i^{-1}x_jx_k^{-1}x_j^{-1}x_ix_jx_k\cdot YY''Y'
\end{equation}
$$
\overset{[x_k,YY''Y']=e}{=} (YY''Y')^{-1}\cdot x_j^{-1}x_i^{-1}x_jx_k^{-1}x_j^{-1}x_ix_j\cdot YY''Y'\cdot x_k.
$$
We now examine the following expression: $x_ix_j\cdot YY''Y'$.
$$
x_ix_j\cdot YY''Y'= x_iYY''x_j[x_j,YY'']Y' = x_iYY''x_jY'[x_j,YY''][[x_j,YY''],Y'] $$$$
\overset{Rel. (\ref{relPointY})}{=} x_jY'x_iYY''[x_j,YY''][[x_j,YY''],Y'] =
x_jY'x_iYY''[x_j,YY'']m,
$$
where we denote $m = [[x_j,YY''],Y']$. Note that $m \in G_3$ and therefore commutes in $G_3/G_4$ with any other element that belongs to $G$. Thus, Equation (\ref{eqnXjik_1}) turns to:
\begin{equation} \label{eqnXjik_2}
x_{j,i,k}^{-1} \equiv  m^{-1}[x_j,YY'']^{-1}Y''^{-1}Y^{-1} x_i^{-1} Y'^{-1} x_j^{-1} \cdot x_jx_k^{-1}x_j^{-1} \cdot x_jY'x_iYY''[x_j,YY''] m x_k \end{equation}
$$ \equiv [x_j,YY'']^{-1}Y''^{-1}Y^{-1} x_i^{-1} Y'^{-1}x_k^{-1}Y'x_iYY''[x_j,YY'']x_k.
$$
Since $[x_k,x_i]  = [x_k,Y] = [x_k,Y'] = [x_k,Y''] = e$,  then Equation (\ref{eqnXjik_2}) becomes:
$$
x_{j,i,k}^{-1} \equiv [x_j,YY'']^{-1}x_k^{-1}[x_j,YY'']x_k.
$$
Denoting $\y = YY''$, we get:
$$
x_{j,i,k}^{-1} \equiv  x_{j,\y,k} \overset{Eqn. (\ref{eqnXijk})}{\equiv} x^{-1}_{\y,j,k}.
$$
Thus:
\begin{equation} \label{eqnXjik_3}
x_{j,i,k} \equiv x_{\y,j,k} = x_j^{-1} \y^{-1} x_j\y x_k^{-1}\y^{-1}x_j^{-1}\y x_jx_k \overset{x_{k,\y} = e,\, Conj.\,\,by\,\,Y'x_i}{\equiv}\end{equation}
$$
(Y'x_i)^{-1}\cdot x_j^{-1} \y^{-1} x_j x_k^{-1}  x_j^{-1} \y x_j x_k \cdot Y'x_i
\overset{x_{k,Y'} = x_{k,i} = e}{=} (Y'x_i)^{-1} x_j^{-1} \y^{-1} x_j x_k^{-1}  x_j^{-1} \y x_j  Y'x_i x_k.
$$
\noindent
Now, $ \y x_j Y' x_i  \overset{Rel. (\ref{relPointY})}{=} x_j Y' x_i \y$.
Therefore,  Equation (\ref{eqnXjik_3}) becomes:
$$
x_{j,i,k} \equiv (x_j Y' x_i \y)^{-1} \cdot x_j x_k^{-1}  x_j^{-1} \cdot x_j Y' x_i \y x_k = \y^{-1} x_i^{-1} Y'^{-1}  x_k^{-1}  Y' x_i \y x_k = e,
$$
when in the last equality we use the fact that $x_k$ commutes with all the other terms in the product.

 \medskip

\textbf{Case (2)}: Let us now check when $[x_k,y_0] \neq e$, when $y_0 \in \{Y,Y',Y''\}$. This happens when the generator $x_k$ participates in a multiple relation $R_3$, where $R_3 \neq R_2$, and one of the generators in the set $\{x_v\}_{v \in \mathcal{Y}}$, where
 $\mathcal{Y} = \{y_0,\ldots,y_k\} \cup \{y'_0,\ldots,y'_{k'}\} \cup \{y''_0,\ldots,y''_{k''}\}$,
 also participates in $R_3$ (see Figure \ref{Case2n}(a)). Note that $x_k$ participates in (at least) two multiple relations: $R_2$ and $R_3$.

 \begin{figure}[h!]
\epsfysize 4cm
\epsfbox{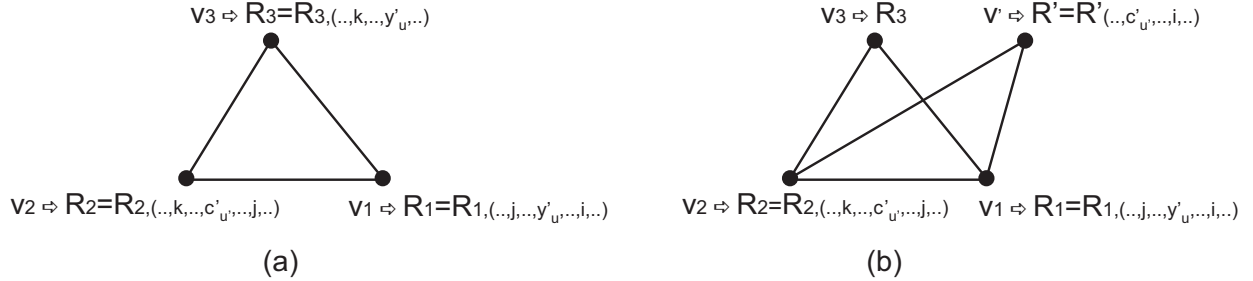}
\caption{Part (a): a part of the graph $Gr(G)$:
The generators $x_i, x_j$ participate in a multiple relation $R_1$ (to whom the vertex $v_1$ is associated) and the generator $x_k$
 participates in two multiple relations $R_2, R_3$ (to whom the vertices $v_2,v_3$ are associated, when we assume that $k>j$). The generator $x_{y'_u}$ is another generator participating in $R_1$ and $R_3$, and $x_{c'_{u'}}$ is another generator participating in $R_2$. Part (b): a forbidden case when $x_{c'_{u'}}$ and $x_i$ participate in the same multiple relation $R'$. In this case, $\beta(Gr(G))>1$.}
\label{Case2n}
\end{figure}

Assume that $k>j$ (the proof for $k<j$ is similar and we leave it to the reader; note that if $k=j$ then $x_{j,i,j}$ is a local generator).
Examining closely $R_2$, we see that:

$$
R_2 = R_{2,(c''_{k''},\ldots, c''_0,k,c'_{k'},\ldots,c'_0,j,c_k,\ldots,c_0)}.
$$

\noindent
(when we take into consideration Remark \ref{remEmptySet} with respect to the sets $\{c_i\}, \{c'_{i'}\}, \{c''_{i''}\}$).
Denote \\$C=x_{c_k}\cdot \ldots \cdot x_{c_0}, C' = x_{c'_{k'}}\cdot \ldots \cdot x_{c'_0}, C'' = x_{c''_{k''}}\cdot \ldots \cdot x_{c''_0}$.
By Definition \ref{defConjFree}, part of the relations of $R_2$ are the following:

\begin{equation} \label{relPointC}
C''x_kC'x_jC = x_kC'x_jCC'' = C'x_jCC''x_k = x_jCC''x_kC' = CC''x_kC'x_j.
\end{equation}

By the Hall-Witt identity in $G_3/G_4$, we get that $[[x_j,x_i],x_k]\cdot [[x_i^{-1},x_k^{-1}],x_j]\cdot[[x_k,x_j^{-1}],x_i^{-1}] \equiv e$.
As $[x_i,x_k]=e$, we get that $[[x_j,x_i],x_k]\cdot[[x_k,x_j^{-1}],x_i^{-1}] \equiv e$. It is easy to see that in $G_3/G_4$,
$[[x_k,x_j^{-1}],x_i^{-1}] \equiv [[x_j,x_k],x_i]^{-1}$, and thus $x_{j,i,k} = [[x_j,x_i],x_k] \equiv [[x_j,x_k],x_i] = x_{j,k,i}$.
Since $x_{j,k,i} \overset{Eqn. (\ref{eqnXijk})}{\equiv} x_{k,j,i}^{-1}$, we prove that $x_{k,j,i} \equiv e.$

\begin{equation} \label{eqnXkji_1}
x_{k,j,i} = x_j^{-1}x_k^{-1}x_jx_kx_i^{-1}x_k^{-1}x_j^{-1}x_k x_j x_i \overset{x_{k,i}=e,\,Conj.\,\,by\,\,C'CC''}{\equiv}\end{equation}$$
(C'CC'')^{-1}\cdot x_j^{-1}x_k^{-1}x_jx_i^{-1}x_j^{-1}x_k x_j x_i \cdot C'CC''.$$

Let $\mathcal{C} = \{c_0,\ldots,c_k\} \cup \{c'_0,\ldots,c'_{k'}\} \cup \{c''_0,\ldots,c''_{k''}\}$.
Note that $x_i$ commutes with every generator in the set $\{x_v\}_{v \in \mathcal{C}}$, since otherwise $\beta(Gr(G))$ would be greater than $1$. Indeed, if both $x_i$ and $x_{c}$ participated in the same multiple relation $R'$ for an index $c \in \mathcal{C}$, then $Gr(G)$ would have (at least) two cycles (see Figure \ref{Case2n}(b)). Therefore $[x_i,C] = [x_i,C'] = [x_i,C''] =  e$ and

\begin{equation}\label{eqnXkji_11}
x_{k,j,i} \equiv (C'CC'')^{-1}\cdot x_j^{-1}x_k^{-1}x_jx_i^{-1}x_j^{-1}x_k x_j  \cdot C'CC'' \cdot x_i.
\end{equation}
Let us examine the expression $x_k x_j  C'CC''$:

$$x_k x_j  C'CC'' =  x_k  C' x_j [x_j,C'] CC''  \overset{t \doteq [[x_j,C'],CC'']}{=}$$
$$
 x_k  C' x_j CC''[x_j,C'] t  \overset{Rel. (\ref{relPointC})}{=}  x_j CC'' x_k C' [x_j,C'] t .
$$
Note that $t \in G_3$ and thus, in $G_3/G_4$, it commutes with any other element.
Therefore, Equation (\ref{eqnXkji_11}) becomes:

\begin{equation} \label{eqnXkji_2}
x_{k,j,i} \equiv t^{-1} [x_j,C']^{-1} C'^{-1} x_k^{-1} C''^{-1} C^{-1} x_j^{-1} \cdot  x_jx_i^{-1}x_j^{-1}\cdot  x_j C C'' x_k C' [x_j,C'] t x_i
\end{equation}
$$
\equiv [x_j,C']^{-1} C'^{-1} x_k^{-1} C''^{-1} C^{-1}  x_i^{-1}  C C'' x_k C' [x_j,C']  x_i = [x_j,C']^{-1}  x_i^{-1}[x_j,C']  x_i,
$$
when in the last equality we used that $[x_i,x_k] = [x_i,C] = [x_i,C'] = [x_i,C''] =  e$.
Thus: $$x_{k,j,i} \overset{Eqn.\,(\ref{eqnXkji_2})}{\equiv} x_{j,C',i} \overset{Eqn.\,(\ref{eqnXijk})}{\equiv} x_{C',j,i}^{-1} \Rightarrow$$
$$
x_{C',j,i} = x_j^{-1}C'^{-1}x_jC' x_i^{-1} C'^{-1} x_j^{-1} C' x_j x_i \overset{[x_i,C']=e,\,Conj.\,\,by\,\,CC''x_k}{\equiv}$$$$
(CC''x_k)^{-1} \cdot x_j^{-1}C'^{-1}x_j  x_i^{-1}  x_j^{-1} C' x_j x_i \cdot CC''x_k =
$$
$$
\overset{[x_i,CC'']=[x_i,x_k]=e}{=} (x_k^{-1}C''^{-1}C^{-1} x_j^{-1}C'^{-1})x_j  x_i^{-1}  x_j^{-1} (C' x_j  CC''x_k)  x_i = e,
$$
\noindent
where in the last equality we used the relation:
$$
 C' x_j   CC'' x_i x_k \overset{Rel. (\ref{relPointC})}{=} x_j CC''kC'
$$
and then the fact that $x_i$ commutes with $C,C',C''$ and $x_k$.

\medskip
\noindent
\textbf{Second case}: In this case $x_k, x_j$ commute, i.e. $x_{j,k}=e$, and $x_k, x_i$ participate in the same multiple relation $R_4$, whose associated vertex in $Gr(G)$ is denoted by $v_4$.
The proof that $x_{j,i,k} \equiv e$ is similar to the proof presented in the First case. Thus, we only outline the main steps and leave the detailed checks to the reader. We use the same notations, regarding the relation $R_1$ and generators participating in it, that appeared in the First case.
Again, we  consider two cases: either that $[x_k,Y] = [x_k,Y'] = [x_k,Y''] = e$ or that at least one of these equalities is not true.

 \medskip

\textbf{Case (1)}: Assuming that $[x_k,Y] = [x_k,Y'] = [x_k,Y''] = e$, we conjugate $x_{j,i,k}$ by $Y'YY''$ and get, using Relation (\ref{relPointY}), that
$x_{j,i,k} \equiv x_{i,Y',k}$. Thus $x_{j,i,k} \equiv x_{Y',i,k}^{-1}$. Conjugating $x_{Y',i,k}$ by $YY''x_j$ and again using Relation (\ref{relPointY}), we see that $x_{Y',i,k} \equiv e$.

\textbf{Case (2)}: This case happens when the generator $x_k$ participates in another multiple relation $R_3$, where $R_4 \neq R_3$, and in $R_3$ participate also one of the generators in the set $\{x_v\}_{v \in \mathcal{Y}}$, where
$\mathcal{Y} = \{y_0,\ldots,y_k\} \cup \{y'_0,\ldots,y'_{k'}\} \cup \{y''_0,\ldots,y''_{k''}\}$
 (see Figure \ref{Case4n}).

 \begin{figure}[h!]
\epsfysize 4.5cm
\epsfbox{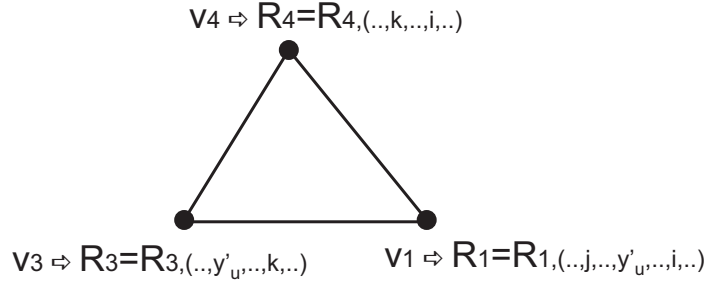}
\caption{
a portion of the graph $Gr(G)$:
The generators $x_i, x_j$ participate in a multiple relation $R_1$ (to whom the vertex $v_1$ is associated) and the generator $x_k$
 participates in two multiple relations $R_3, R_4$ (to whom the vertices $v_3,v_4$ are associated). The generator $x_{y'_u}$ is  another generator participating in $R_1$ and $R_3$.}
\label{Case4n}
\end{figure}

\noindent
Note that it is known that $i<k$. Examining closely $R_3$, we see that:

$$
R_3 = R_{3,(z''_{k''},\ldots, z''_0,k,z'_{k'},\ldots,z'_0,i,z_k,\ldots,z_0)}.
$$
(when we take into consideration Remark \ref{remEmptySet} with respect to the sets $\{z_i\}, \{z'_{i'}\}, \{z''_{i''}\}$).
Denote\\ $Z = x_{z_k}\cdot \ldots \cdot x_{z_0}, Z' = x_{z'_{k'}}\cdot \ldots \cdot x_{z'_0}, Z'' = x_{z''_{k''}}\cdot \ldots \cdot x_{z''_0}$.
Therefore, by Remark \ref{RemRelCyclic}, part of the relations of $R_3$ are the following:

\begin{equation} \label{relPointZ}
Z''x_kZ'x_iZ = x_kZ'x_iZZ'' = Z'x_iZZ''x_k = x_iZZ''x_kZ' = ZZ''x_kZ'x_i.
\end{equation}
\noindent
Again, by the Hall-Witt identity in $G_3/G_4$, we  get that $x_{j,i,k} \equiv x^{-1}_{i,k,j}$ and therefore $x_{j,i,k} \equiv x_{k,i,j}$.
Noting that $[x_j,Z] = [x_j,Z'] = [x_j,Z''] = e$ (otherwise $\beta(Gr(G))>1$), we conjugate $x_{k,i,j}$ by $Z'ZZ''$, and using Relation (\ref{relPointZ}), we get that
$x_{k,i,j} \equiv x_{i,Z',j}$. Thus $x_{k,i,j} \equiv x_{Z',i,j}^{-1}$. Conjugating $x_{Z',i,j}$ by $ZZ''x_k$ and again using Relation (\ref{relPointZ}), we get that $x_{Z',i,j} \equiv e$.

\medskip
\noindent

\textbf{Third case}: In this case $x_{i,k} \neq e$ and $x_{j,k} \neq e$. There are two cases when this  may occur: the first is when all of the three generators
$x_i,x_j,x_k$ participate in the same multiple relation. In this case, the generator $x_{j,i,k} \in G_3/G_4$ is a local
generator, so we are done. The second case, of a cycle of length $3$ in the graph $Gr(G)$, occurs when the vertices of the cycle correspond to the different \emph{multiple} relations $R_1,R_2$ and $R_4$, where $x_i,x_j$ participate in $R_1$, $x_k,x_j$ participate in $R_2$ and
$x_i,x_k$ participate in $R_4$ (see Figure \ref{Case3n}(a)).

 \begin{figure}[h!]
\epsfysize 4cm
\epsfbox{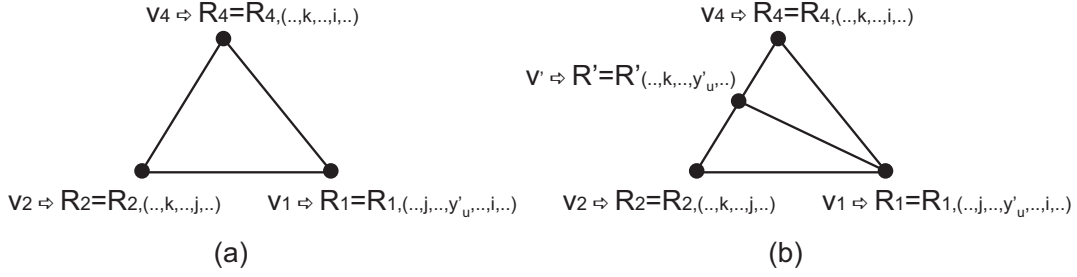}
\caption{Part (a): A cycle of length three: all pairs of generators participate in different multiple relations. The generator $x_{y'_u}$ participates in $R_1$. Part (b): a forbidden case when $x_{y'_u}$ and $x_k$ participate in the same multiple relation $R'$. In this case, $\beta(Gr(G))>1$.}
\label{Case3n}
\end{figure}

\noindent
We use the notations of the First case, i.e. that
$$
R_1 = R_{1,(y''_{k''},\ldots, y''_0,j,y'_{k'},\ldots,y'_0,i,y_k,\ldots,y_0)}.
$$
We also use the notations $Y,Y',Y''$. Note that $[x_k,Y] = [x_k,Y'] = [x_k,Y'']  = e$, as otherwise $x_k$ would participate in (at least) three multiple relations
and $\beta(Gr(G)) > 1$ (see Figure \ref{Case3n}(b)).
Thus:

\begin{equation} \label{eqnThirdCase_1}
x_{j,i,k} = x_i^{-1}x_j^{-1}x_ix_jx_k^{-1}x_j^{-1}x_i^{-1}x_jx_ix_k \overset{Conj.\,\,by\,\,Y'YY'',\,[x_k,Y'YY'']=e}{\equiv}\end{equation}
$$
\equiv (Y'YY'')^{-1}x_i^{-1}x_j^{-1}x_ix_jx_k^{-1}x_j^{-1}x_i^{-1}x_jx_i Y'YY''  x_k.
$$
\noindent
Let us examine the expression $x_jx_i Y'YY''$.
$$
x_jx_i Y'YY'' = x_j Y' x_i [x_i,Y'] YY'' = x_j Y' x_i YY'' [x_i,Y'][[x_i,Y'],YY''] \overset{Rel. (\ref{relPointY})}{=} x_i YY'' x_j Y' [x_i,Y'] r,
$$
where we set $r = [[x_i,Y'],YY''] \in G_3$, which commutes in $G_3/G_4$ with every other element. Thus Equation (\ref{eqnThirdCase_1}) becomes:

\begin{equation} \label{eqnThirdCase_2}
x_{j,i,k} \equiv   r^{-1}[x_i,Y']^{-1}Y'^{-1} x_j^{-1} Y''^{-1}Y^{-1} x_i^{-1}   \cdot x_ix_jx_k^{-1}x_j^{-1}x_i^{-1} \cdot x_i YY'' x_j Y' [x_i,Y'] r x_k \equiv \end{equation} $$ [x_i,Y']^{-1}Y'^{-1} x_j^{-1} Y''^{-1}Y^{-1}  x_jx_k^{-1}x_j^{-1}  YY'' x_j Y' [x_i,Y']  x_k.$$

\noindent
Noting that $[x_i,Y']  x_k = x_k [x_i,Y'] [[x_i,Y'],x_k] \overset{x_{i,Y',k}\in G_3}{\equiv} x_k [[x_i,Y'],x_k]  [x_i,Y'],$ we
denote  $$f = Y'^{-1} x_j^{-1} Y''^{-1}Y^{-1}  x_jx_k^{-1}x_j^{-1}  YY'' x_j Y' x_k,$$
\noindent
and get from Equation (\ref{eqnThirdCase_2}) that:
\begin{equation} \label{eqnThirdCase_3}
x_{j,i,k} \equiv   (f \cdot  [[x_i,Y'],x_k])^{[x_i,Y']} \overset{x_{j,i,k \in G_3,\,Rem.\, \ref{remConj}}}{\equiv} f \cdot  [[x_i,Y'],x_k].
\end{equation}

\noindent
Since $x_{j,i,k},[[x_i,Y'],x_k] \in G_3/G_4$, then also $f \in G_3/G_4$. Denoting $\y = YY''$, we examine the expression $f$:
$$
f \overset{Conj.\,\,by\,\,x_i}{\equiv} x_i^{-1} Y'^{-1} x_j^{-1} \y^{-1}  x_jx_k^{-1}x_j^{-1}  \y x_j Y' x_k x_i =
x_i^{-1} Y'^{-1} x_j^{-1} \y^{-1}  x_jx_k^{-1}x_j^{-1} \cdot  \y x_j Y'   x_i \cdot x_k [x_k,x_i].
$$
\noindent
Now,
$$
 \y x_j Y' x_i  \overset{Rel. (\ref{relPointY})}{=}
x_j Y' x_i \y.
$$
Thus,
$$
f \equiv \y^{-1}  x_i^{-1} Y'^{-1} x_j^{-1} \cdot x_jx_k^{-1}x_j^{-1} \cdot x_j Y' x_i \y x_k [x_k,x_i] =
   \y^{-1}  x_i^{-1} Y'^{-1}  x_k^{-1}  Y' x_i \y x_k [x_k,x_i] \overset{[x_k,Y'] = [x_k,\y] = e}{=} $$$$
    \y^{-1}  x_i^{-1}   x_k^{-1}   x_i  x_k \y [x_k,x_i] = \y^{-1} [x_k,x_i]^{-1}\y [x_k,x_i] = x^{-1}_{k,i,\y}.
$$
Moreover, using the Hall-Witt identity in $G_3/G_4$ (as in the First case), we see that $[[x_i,Y'],x_k] \equiv [[x_i,x_k],Y']$ and thus
$[[x_i,Y'],x_k] \equiv x^{-1}_{k,i,Y'}$.
Therefore, Equation  (\ref{eqnThirdCase_3}) becomes
$$
x_{j,i,k} \equiv  x^{-1}_{k,i,\y} \cdot x^{-1}_{k,i,Y'}.
$$
By Proposition \ref{prsG3G4}(II) (see Equation (\ref{eqnEqG3G4})), the expression on the right hand side is a product of inverses of terms of the form $x_{k,i,y}$, where
$y \in \mathcal{Y} \doteq \{y_i\} \cup \{y'_{i'}\} \cup \{y''_{i''}\}$. However, as $x_k$ commutes with $x_y$ (for every $y \in \mathcal{Y}$), this case is already handled in the  Second case.
Therefore $x_{j,i,k} \equiv e$.
\end{proof}

\begin{remark}\label{remPhi3For}
\emph{Obviously there are cyclic-related groups, which may not be conjugation-free, whose $\phi_3$ is not bounded by the above upper bound. For example, for the pure braid group $BP_4$  it is known that $\phi_3 = 10$ (the ranks of the lower central series of the pure braid group were first studied by Kohno \cite{K2}). Moreover, it has a cyclic-related presentation, under which it is generated by $6$ generators with $4$ cyclic relations of length $3$ and $3$ cyclic relations of length $2$ (see \cite{MM}). However, were $BP_4$ a conjugation-free group with $\beta(Gr(PB_4)) \leq 1$, then $\phi_3 \leq 8$.}
\end{remark}

 \begin{remark} \label{remRelCond}
\rm{ The restriction on the conjugation-free group $G$ in Theorem \ref{ThmMain}, i.e. that \linebreak $\beta(Gr(G))~\leq~1$, can be relaxed. Assume that $\beta(Gr(G))=1$
and that the length of the cycle is $t$. Let $v_1,\ldots,v_t$ be the vertices of the graph on the cycle, corresponding to the multiple relations
$R_1,\ldots, R_t$. Let $e_{i_1}, e_{i_2}, \ldots, e_{i_t}$ be the edges of the cycle, when $e_{i_j}$ connects the vertices $v_j$ and $v_{j+1(\text{mod }t)}$.
Every edge $e_{i_j}$, $1 \leq j \leq t$ corresponds to a generator $x_{i_j}$ which participates both in $R_j$ and in $R_{j+1(\text{mod }t)}$. Let $X_0 = \{x_{i_1},\ldots, x_{i_t}\}$. For each relation $R_i$, $2 \leq i \leq t$, let $X_{p_i}$ be the set of generators of $G$ participating in $R_i$, and let $X_2 = \cup_{i=2}^t X_{p_i}$. Denote  $X_1 = X_2 - X_0$.

  By examining the proof of Theorem \ref{ThmMain}, note  that if a generator $x_r$ participates, for example, in $R_1$, then in order to prove that a non-local generator $x_{j,i,k}$ is equivalent to $e$, when one of the indices $j,i$ or $k$ is $r$, we only have to demand that $x_r$ commutes with all the generators in $X_1$.

  This means that in order that Theorem \ref{ThmMain} could be applied, we can omit the requirement that $\beta(Gr(G)) \leq 1$ and require that if there are a few cycles (where $v_1$ is a vertex, corresponding to the relation $R_1$, in a cycle composed of the vertices $v_1,\ldots,v_t$ and $x_r$ participates in $R_1$), then $x_r$ can participate in other multiple relations $R$ such that $R \neq R_i$, $1 \leq i \leq t$, as long as in this multiple relation $R$, there are no other generators,  belonging to $X_1$, which participate in it.
  }
 \end{remark}

 \begin{definition}
\emph{ A planar, simple connected graph $H$ is called \emph{cycle-separated} if  every two cycles of $H$ do not have a vertex in common and all the edges of $H$ are straight segments.
} \end{definition}
\noindent
Obviously, a graph $H$ with $\beta(H) \leq 1$ is cycle-separated. An example to a cycle-separated graph with two cycles is presented in Figure \ref{cycleSep}(a).

 \begin{figure}[h!]
\epsfysize 2cm
\epsfbox{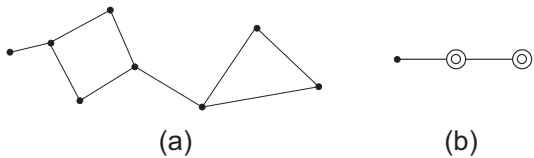}
\caption{Part (a): An example to a cycle-separated graph $H$. Part (b): The graph $H$ after the identification of its cycles into cycled-vertices.}
\label{cycleSep}
\end{figure}

\begin{remark} \label{remGenFormCycSep}
\rm{Let us find the general form of a cycle-separated graph. Given a cycle in a cycle-sparated graph, we identify all its vertices and edges connecting them and call the resulting object a \emph{cycled-vertex}, depicted by $\circledcirc$ (in contrast to ordinary vertices, which are depicted by $\bullet$). For example, after this identification, the graph $H$ in Figure \ref{cycleSep}(a) is transformed to the graph presented in Figure \ref{cycleSep}(b).

We claim that given a cycle-separated graph, after this identification, the new graph will be a tree (whose vertices are either  $\circledcirc$ or $\bullet$). Indeed, if there is a cycle in the new graph, then the original graph would have two cycles with common vertices.
}
\end{remark}

  Thus we have the following:
\begin{prs} \label{prsCycSepG3G4}
If $G$ is a conjugation-free group whose graph is a disjoint
union of
cycle-separated graphs, then
$$
\phi_3(G) \leq \sum_{i \geq 3} n_i\omega_3(i-1).
$$
\end{prs}

\begin{proof}
 The group $G$, whose graph is cycle-separated, satisfies the (relaxed) conditions posed in Remark \ref{remRelCond}, and since $G$ is conjugation-free, Proposition \ref{prsFirstPart} can be applied for the group $G$. By Example \ref{exmGraphCycRel}(3),  this claim is true when $Gr(G)$ is a disjoint union of cycle-separated graphs.
\end{proof}

\section{Line arrangements and conjugation-free graphs} \label{sectionLineArrCFgroup}

In this section we prove that given a conjugation-free group $G$ such that $Gr(G)$ is a cycle-separated graph, $\phi_3(G) = \sum_{i \geq 3} n_i\omega_3(i-1)$, and in fact $\phi_k(G) = \sum_{i \geq 3} n_i\omega_k(i-1)$ for every $k \geq 3$. In order to do so, we have to use tools arising from the theory of line arrangement. For the reader who is not familiar with line arrangements, let us outline the main constructions of this section.

 In section \ref{subsecAssocLineArr} we associate for every conjugation-free group $G$ (with an associated cycle-separated graph) a line arrangement $\LL(G)$ in $\CC^2$ (see Definition \ref{defLineArrAss}), later showing in Section \ref{subsecPi1CompLine} that the fundamental group $\pi_1(\CC^2 - \LL(G))$ is isomorphic to $G$. In Section \ref{subsecLCSLineArr} we see that given any line arrangement $\LL$, it is known by Falk \cite{Falk2} that $\phi_3(\pi_1(\CC^2 - \LL)) \geq \sum_{i \geq 3} n_i\omega_3(i-1)$, where $n_i$ is the number of the singular points of $\LL$ with multiplicity $i$. Combining this with Proposition \ref{prsCycSepG3G4}, we get that  $\phi_3(G) = \sum_{i \geq 3} n_i\omega_3(i-1)$. Finally, Papadima-Suciu \cite[Theorem 2.4]{PS} shows that if $\phi_3$ attains this equality, then $\phi_k$ attains it for every $k \geq 3$.

 \subsection{The associated line arrangement} \label{subsecAssocLineArr}

An {\it affine line arrangement} in $\CC^2$ is a union of copies of $\C^1$ in $\C^2$. Such an arrangement is called {\em real} if the defining equations of all its lines are written with real coefficients, and {\em complex} otherwise. A \emph{multiple point} $p \in \LL$ is an intersection point where more than two lines intersect (i.e. an intersection point with intersection multiplicity greater than $2$).

 We construct an affine  line arrangement, associated to a conjugation-free group.

 \begin{definition} \label{defLineArrAss}
 \rm{ Let $G$ be a conjugation-free group, generated by $x_1,\ldots,x_n$, with a planar graph $Gr(G)$ such that all its edges are straight segments. We assume that $Gr(G)$ is connected (see Remark \ref{remSevComp} for the case when $Gr(G)$ has several components). Let $v_1,\ldots,v_m$ be the set of vertices of $Gr(G)$, corresponding to the set of multiple relations $R_1,\ldots,R_m$ of $G$. We construct a real line arrangement $\LL(G)$ from $Gr(G)$.

 If $Gr(G)$ is empty, then  we define $\LL(G)$ as a real generic line arrangement (that is, all the singular points are nodes), consisting of $n$ lines.

 Hence assume that $Gr(G)$ is not an empty graph. If $Gr(G)$ consists of only one vertex $v_1$ and no edges, then $v_1$ corresponds in $G$ to a multiple relation of length $t$. In that case, fix a point $p_1 \in \RR^2$ and pass through it $t$ different real lines.

  Assume thus that there is at least one edge in $Gr(G)$. First, we extend the edges of $Gr(G)$ to be infinite straight lines, and call the union of all these lines $\LL$. Let $I$ be the set of intersection points of $\LL$, i.e., at this step, $I = \{p_1,\ldots,p_m\}$, where the $p_i$'s are the points corresponding to the vertices of $Gr(G)$ (for an example, see Figures \ref{GraphToLine}(a) and  \ref{GraphToLine}(b.I)). We now build the rest of the arrangement inductively.

 Looking at the vertex $v_1$, it corresponds in $G$ to a multiple relation $R_1 = R_{1,(i_t,\ldots,i_1)}$, $t \geq 3$, and in $\LL$ to the singular point $p_1$. $Gr(G)$ is connected, so there is an edge $e$ connecting $v_1$ with another vertex. Renumerate the generators if needed, this edge corresponds to a generator $x_{i_1}$, which participates in the relation  $R_1$ and in another multiple relation; this generator corresponds to the  already drawn line $\ell_{i_1} \in \LL$ (this is the extended edge $e = e_{i_1}$). In addition, there might be other generators $x_{i_2},\ldots,x_{i_u}$, for which the corresponding lines were already drawn (see Figure  \ref{GraphToLine}(b.I)).
  Thus, assume that for the generators $x_{i_{u+1}},\ldots,x_{i_t}$, the corresponding lines were not drawn. We draw a line $\ell_{i_{u+1}}$ which corresponds to the generator $x_{i_{u+1}}$, in the following way: the line $\ell_{i_{u+1}}$ passes through the singular point $p_1$ and intersects all the other lines in $\LL$ at nodes, i.e. it does not pass through the singular points in $I - \{p_1\}$. By abuse of notation, let $\LL$ be the new arrangement $\LL \cup \ell_{i_{u+1}}$ and let $I$  be the set of all intersection points of the new arrangement $\LL$. We now repeat the same process for all the other generators $x_{i_{u+2}},\ldots,x_{i_t}$ (that is, the lines $\ell_{i_{u+2}},\ldots,\ell_{i_t}$ pass through $p_1$ and intersect all the lines of $\LL$ at nodes). In this way, we have attached to every generator that participates in $R_1$ a line in $\LL$ (see Figure \ref{GraphToLine}(b.II)).

  We now look at the other vertices which are neighbors of $v_1$. If $v_2$ is such a vertex (corresponding to the singular point $p_2$ in $\LL$), corresponding to the relation $R_2 = R_{2,(j_s,\ldots,j_1)}$ in $G$, we know that  already for some of the generators in $\{x_{j_k}\}_{k=1}^s$, the corresponding lines were drawn. Thus, for the rest of the generators in this set, we draw the corresponding lines, in the same way as above: these  lines pass only through the singular point $p_2$ and intersect all the other lines in $\LL$ at nodes (see Figure \ref{GraphToLine}(III)).

  In the same way, we go over all the vertices of $Gr(G)$ (see Figure \ref{GraphToLine}(b.IV)). Eventually we get that $\LL$ is a real line arrangement in $\RR^2$ with $n'$ lines. In $n' < n$ then there are $n-n'$ generators of $G$: $x_{k_1},\ldots,x_{k_{n-n'}}$ which do not participate in any multiple relation, that is, $[x_i,x_{k_u}]=e$ for every $1 \leq i \leq n$ and $1 \leq u \leq n-n'$.

  We draw a line $\ell_{k_1}$ which corresponds to the generator $x_{k_{1}}$, in the following way: the line $\ell_{k_1}$ does not pass through any singular point of $\LL$, hence it intersects all the lines of $\LL$ at nodes (see Figure \ref{GraphToLine}(b.V)).  By abuse of notation, let $\LL$ be the new arrangement $\LL \cup \ell_{k_{1}}$. We now repeat the same process for all the other generators $x_{k_{2}},\ldots,x_{k_{n-n'}}$. Eventually we get that $\LL$ is a real line arrangement with $n$ lines.

    Complexifying the equations of the lines, we get an affine real line arrangement $\LL(G)$ in $\CC^2$  with $n$ lines. Note that the number of multiple points of $\LL(G)$ is the number of multiple relations.
 }
 \end{definition}

 \begin{example}
\rm{ Let $H$ be a conjugation-free group, generated by $7$ generators $x_1,\ldots,x_7$ with the following relations:
 $$
 R_1:x_3 x_2 x_1 = x_2 x_1 x_3 = x_1 x_3 x_2, \,\,R_2: x_5 x_4 x_3 = x_4 x_3 x_5 = x_3 x_5 x_4,\,\,  R_3:x_6 x_5 x_1 = x_5 x_1 x_6 = x_1 x_6 x_5,
 $$
 and
 $$
 R_4:x_4 x_2 = x_2 x_4, \,\, R_5:x_5 x_2 = x_2 x_5,\,\, R_6:x_6 x_2 = x_2 x_6,\,\, R_7:x_6 x_4 = x_4 x_6,\, R_{7+j}:[x_7,x_i]=e,\,1 \leq j \leq 6.
 $$
 The graph $Gr(H)$ is presented in  Figure \ref{GraphToLine}(a), where next to each vertex we write the associated  multiple relation. The edges of $Gr(H)$ are denoted by $e_1,e_3$ and $e_5$ and the vertices by $v_1,v_2$ and $v_3$ (recall that the index of $e_i$ corresponds to the common generator $x_i$ that participate in the two relations, that correspond to the vertices connected by $e_i$).

 In order to construct the corresponding line arrangement $\LL(H)$ we first extend the edges $e_1,e_3$ and $e_5$ to infinite straight lines $\ell_1,\ell_3$ and $\ell_5$. These lines intersect each other at the points $p_1,p_2$ and $p_3$, which correspond to the vertices $v_1,v_2$ and $v_3$ (see Figure \ref{GraphToLine}(b.I)). Next, look at the vertex $v_1$, which is associated to the relation $R_{1,(3,2,1)}$. Since the lines
 $\ell_1$ and $\ell_3$ were already drawn, we only have to draw the line $\ell_2$, which correspond to the generator $x_2$. We draw it such that it would pass through $p_1$ and intersect the other lines at nodes (see Figure \ref{GraphToLine}(b.II)). Then we do the same for  the vertices $v_2$ and $v_3$: we first draw the line $\ell_4$ (passing though $p_2$) and finally, the line $\ell_6$ that pass through $p_3$ (see Figure \ref{GraphToLine}(b.III) and \ref{GraphToLine}(b.IV) respectively). We get a line arrangement with $6$ lines, corresponding to the generators $x_1,\ldots,x_6$. As for the generator $x_7$, we draw a generic line $\ell_7$ which intersects the lines $\ell_1,\ldots,\ell_6$ at nodes (see Figure \ref{GraphToLine}(b.V)).

  \begin{figure}[h!]
\epsfysize 8.5cm
\epsfbox{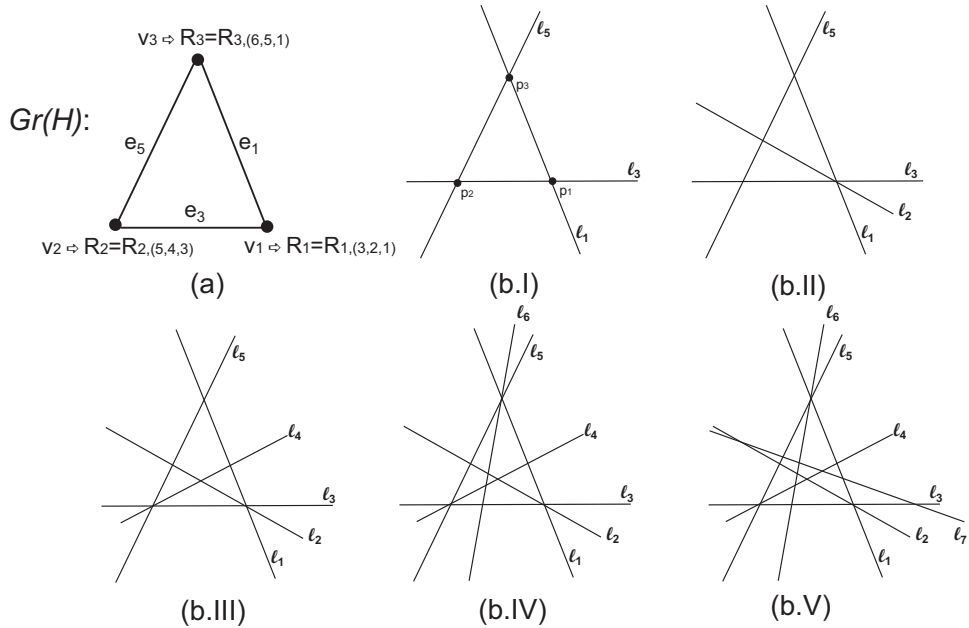}
\caption{An example to a construction of a line arrangement from a conjugation-free group $H$ and its cycle-separated graph $Gr(H)$.}
\label{GraphToLine}
\end{figure}
}
 \end{example}

 \begin{remark} \label{remSevComp}
 \em{
 Let $G$ be a conjugation-free group such that $Gr(G)$ has several components, that is $Gr(G) = \cup_{i=1}^{v}Gr(G)_i$. In order to build $\LL(G)$ we just apply Definition \ref{defLineArrAss} to every connected component $Gr(G)_i$ separately, when the addition of the lines corresponding to generators which do not participate in any multiple relation is done at the end, after we have drawn all the lines that are associated to all the
 components of $Gr(G)$.
 }
 \end{remark}

 \subsection{Fundamental groups of complements of line arrangements} \label{subsecPi1CompLine}

 For a real or complex line arrangement $\mathcal L$, Fan \cite{Fa2} defined a graph $Gr_L(\mathcal{L})$ which is associated to the multiple points of $\LL$. We give here its version for real arrangements (the general version is more delicate to explain and will be omitted): Given a real line arrangement $\mathcal L \subset \CC^2$, the graph $Gr_L(\mathcal{L})$ lies on the real part of $\mathcal L$. Its vertices are  the multiple points of $\mathcal{L}$
 and its edges are  the segments between the multiple points on lines which have at least two multiple points. Note that if the arrangement consists of three
multiple points on the same line, then $Gr_L(\mathcal{L})$ has three vertices on the same edge.
If two such lines happen to intersect in a \emph{simple} point (i.e. a point where exactly two lines are intersected), it is ignored
(i.e. there is no corresponding vertex in the graph). See an example in Figure \ref{graph_GL}.

\begin{figure}[!ht]
\epsfysize 3cm
\centerline{\epsfbox{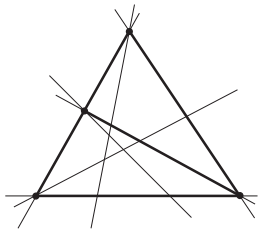}}
\caption{An example for the graph $Gr_L(\mathcal L)$ associated to a line arrangement $\LL$.}\label{graph_GL}
\end{figure}

\medskip

We have now the following obvious proposition:
\begin{prs} \label{prsCFgrToLine}
Let $G$ be a conjugation-free group with a planar graph $Gr(G)$. Let $\LL(G)$ be the associated line arrangement. Then $Gr(G) = Gr_L(\LL(G)) $.
\end{prs}

\begin{proof}
Indeed, through every vertex of $Gr(G)$ we draw more than two lines, and thus every vertex corresponds to a multiple point on $\LL(G)$, which in turn corresponds to a vertex (found in the same position) of $Gr_L(\LL(G))$. The edges connecting these vertices in $Gr(G)$ are only extended to infinite straight lines in $\LL(G)$, now connecting between multiple points, and hence considered again as edges in $Gr_L(G)$.
\end{proof}

\begin{remark} \label{remStrLineArr}
\emph{Given  a complex line arrangement $\mathcal L$, it is known that $\pi_1(\CC^2 - \LL)$ is a free abelian group if and only if $\LL$ has only nodes as intersection points (see for example \cite[Example 1.6(a)]{DOZ}). Moreover, Fan \cite{Fa1,Fa2} proved that if the graph $Gr_L(\mathcal L)$ has no cycles, i.e. $\beta(Gr_L(\LL))=0$, then
$\pi_1 (\CC \PP ^2 - \mathcal L)$, and therefore also $\pi_1 (\CC ^2 - \mathcal L)$, are isomorphic to a direct sum of a free abelian group and free groups.
Eliyahu et al. \cite{ELST} proved the inverse direction to Fan's result, i.e. if the fundamental group of the arrangement is a direct sum of a free abelian group and free groups, then the graph $Gr_L(\mathcal L)$ has no cycles.}
\end{remark}

Recall that for computing the fundamental group of a complement of a curve $C$ in $\C^2$, one can use the Zariski-van Kampen thereom \cite{vK}. This theorem uses a generic projection $\pi : \C^2 \to \ell = \C^1$ (or a projection $\pi : \C\PP^2 \to  \C\PP^1$ with a center $O$) to a generic line $\ell$ in order to induce  in the fiber
$\C^1_{p_0} = \pi^{-1}({p_0})$  generators of the group $\pi_1(\C^1_{p_0} - (\C^1_{p_0} \cap C),{p_0})$, where ${p_0}$ is a generic point on $\ell$. These generators also generate $\pi_1 (\CC^2 - C,{p_0})$.

Using these notations,
we recall the notion of a {\it conjugation-free geometric presentation} for the fundamental group of the complement of line arrangements (see \cite{EGT1,FG}):
\begin{definition}\label{CFGP-LineArr}
Let $G = \pi_1(\CC^2 - \LL)$ be the fundamental group of the affine complement of a real line  arrangement with $n$ lines $\{\ell_1,\ldots,\ell_n\}$. We say that $G$ has {\em a conjugation-free geometric presentation} if $G$ has a presentation with the following properties:
\begin{enumerate}
\item The generators $\{ x_1,\dots, x_{n} \}$ are the meridians of the lines at $\C^1_{p_0}$.
\item The induced relations are of the following type:
$$x_{i_t} x_{i_{t-1}} \cdot \ldots \cdot x_{i_1} = x_{i_{t-1}} \cdot \ldots \cdot x_{i_1} x_{i_t} = \cdots = x_{i_1} x_{i_t} \cdot \ldots \cdot x_{i_2}$$
where $\{ i_1,i_2, \dots , i_t \} \subseteq \{1, \dots, n \}$ is an increasing subsequence of indices.
This relation is the only relation
induced by an intersection point of multiplicity $t$, i.e. the intersection of the lines $\{ \ell_{i_1}, \ell_{i_2}, \ldots , \ell_{i_t} \}$. Note that if $t=2$  we get the usual commutator and also that there are no other relations in $G$.
\end{enumerate}
Note that we claim that with respect to  \emph{particular choices} of the  line $\ell$ (the line to which we project the arrangement), the point $p_0$ (being the basepoint for  the meridians in the fiber $\CC^1_{p_0}$ and also for $\pi_1(\CC^2-\LL)$)
and the projection point $O$, the conjugation-free property holds.
\end{definition}

\begin{definition}
\emph{A line arrangement $\LL$ is called \emph{conjugation-free} if $\pi_1(\C^2 - \LL)$ has  a conjugation-free geometric presentation.
}\end{definition}

\begin{prs} \label{prsLineArrToCFG}
If a line arrangement $\LL$ is conjugation-free then $\pi_1(\C^2 - \LL)$ is a conjugation-free group.
\end{prs}

\begin{proof}
In order to check that $\pi_1(\C^2 - \LL)$ is a conjugation-free group, we check the three requirements listed in Definition \ref{defCyclicRelated}. Requirement (1) (that every relation is a conjugation-free relation, i.e. a cyclic-relation without conjugations) is fulfilled due to Property (2)  of Definition \ref{CFGP-LineArr}. If $\ell_{j_1}$, $\ell_{j_2}$ are two different lines in $\LL$, then they meet only at one point $p$, whose multiplicity is $m(p)$. Therefore, the associated generators $x_{j_1}$, $x_{j_2}$ to these lines participate in the cyclic relation $R$ of length $m(p)$, induced from the point $p$, which is the only relation in which both of these generators participate. Thus requirement (2) of Definition \ref{defCyclicRelated} is fulfilled. As for requirement (3), assume that there are two relations $R_{p,(i_t,\ldots,i_1)}$, $R_{p',(j_s,\ldots,j_1)}$ of $\pi_1(\CC^2 - \LL)$, induced from the singular points $p,p'$,
 such that: $$|\{ i_1,i_2, \dots , i_t \} \cap \{ j_1,j_2, \dots , j_s \}|  \geq 2 .$$ Let $\{i',i''\} \subseteq \{ i_1,i_2, \dots , i_t \}$ be two of the indices contained in this intersection. As every relation is induced from an intersection point, this means that the lines $\ell_{i'}$ and $\ell_{i''}$ intersect at two different points, which is impossible.
\end{proof}

\begin{remark}\label{RemRelCyclic}
\rm{
(1) Given any line arrangement $\LL$,  The Zariski-van Kampen thereom shows that an intersection point $p$ of multiplicity $t$ induces the following relation in $\pi_1(\C^2 - \LL)$:
$$x_{i_t}^{s_{p,t}} x_{i_{t-1}}^{s_{p,t-1}} \cdot \ldots \cdot x_{i_1}^{s_{p,1}} = x_{i_{t-1}}^{s_{p,t-1}} \cdot \ldots \cdot x_{i_1}^{s_{p,1}} x_{i_t}^{s_{p,t}} = \cdots = x_{i_1}^{s_{p,1}} x_{i_t}^{s_{p,t}} \cdot \ldots \cdot x_{i_2}^{s_{p,2}},$$
where the $x_{i_j},\,1 \leq j \leq t$ are the meridians of the lines $l_{i_j}$ and $s_{p,i} \in \langle x_1,\ldots,x_n \rangle$. Therefore, for every line arrangement $\LL$, the group $\pi_1(\CC^2-\LL)$ is a cyclic-related group
(the proof of that is along the same lines as the proof of Proposition \ref{prsLineArrToCFG}).

(2) Note  that if for every singular point $p$ in the line arrangement $\LL$, $s_{p,i}=e$ (in the induced relation) for every $i$, then $\LL$ is conjugation-free (and thus $\pi_1(\C^2 - \LL)$  is a conjugation-free group).

(3) The Zariski-Lefschetz hyperplane section theorem (see \cite{milnor})
states that
$\pi_1 (\PP^N - S) \simeq \pi_1 (H - (H \cap S)),$
where $S$ is a hypersurface and $H$ is a generic 2-plane.
When $S$ is a hyperplane arrangement, $H \cap S$ is a line arrangement in $\PP^2$. Thus, one can investigate the topology of hyperplane arrangements via the fundamental groups $\pi_1(\PP^2- \LL)$ and $\pi_1(\CC^2 - \LL)$, where $\LL$ is an arrangement of lines.
}
\end{remark}

We can now state the opposite statement to Proposition \ref{prsCFgrToLine}:
\begin{prs} \label{prsIsoGraphs}
Let $\LL$ be a conjugation-free real line arrangement, $G = \pi_1(\C^2 - \LL)$. Then $Gr_L(\LL) \simeq Gr(G)$.
\end{prs}

\begin{proof}
Indeed, the vertices of $Gr_L(\LL)$ correspond to multiple points, which induce multiple conjugation-free relations in $G$, and thus correspond to vertices in $Gr(G)$. The edges of $Gr_L(\LL)$ correspond to lines connecting two multiple points, which correspond in $G$ to generators
participating in the relations (which correspond to these points), and hence these generators correspond to edges in $Gr(G)$ connecting the
 corresponding vertices there.
\end{proof}

Note that the importance of the family of line arrangements whose fundamental group has a conjugation-free geometric presentation is that such a presentation of the fundamental group can be read directly from the arrangement (in particular, from its intersection lattice) without any additional computations, and thus depends on the combinatorics of the arrangement.

\medskip

Denote by $\beta (\mathcal L)$  the first Betti number of the graph $Gr_L(\mathcal L)$.
We recall the following proposition from \cite{EGT2} (whose complete proof is given at \cite{FG}):

\begin{prs} \label{prsCFless1}
 Let $\LL$ be a real line arrangement with $\beta(\LL) \leq 1$. Then $\LL$ is  conjugation-free.
\end{prs}

The class of conjugation-free real line arrangement is in fact bigger than the class of line arrangements with $\beta(\LL) \leq 1$. We cite the following definition from \cite[Definition 6.4]{FG}:

\begin{definition} \label{defCFG}  \emph{
Let $G$ be a planar connected graph and denote by deg$(v)$ the number of edges exiting from a vertex  $v \in G$. The graph $G$ is called a \emph{conjugation-free graph} if:
\begin{enumerate}
\item $\beta(G) \leq 1$, or
\item Let $\{v_i\}_{i=1}^m$ be the set of vertices in $G$ satisfying deg$(v_i) \leq 2$. For each $v_i$, $1 \leq i \leq m$, denote by $V_i$ the subset of $G$, composed of the vertex $v_i$ and the edge(s) exiting from it. Let $X = X(G) \doteq \bigcup_{i=1}^m V_i$. Then $G$ is a conjugation-free graph if $G - X$ is.
\end{enumerate}
}
\end{definition}

\noindent
For example, the graph in Figure \ref{cycleSep}(a) is a conjugation-free graph with two cycles. In \cite{FG} the following result was proved:

\begin{thm} \label{thmCFG} \cite[Theorem 6.5]{FG}
Let $\LL$ be a real line arrangement. If $Gr_L(\LL)$ is a disjoint union of conjugation-free graphs, then $\LL$ is a conjugation-free line arrangement.
\end{thm}

\begin{remark}
\emph{By looking at the general form of a cycle-separated graph (see Remark \ref{remGenFormCycSep}), it is easy to see that a cycle-separated graph is a conjugation-free graph: since if $G$ is a cycle-separated graph, $X = X(G)$ as in Definition \ref{defCFG}, then $\beta(G-X) \leq 1$.
}\end{remark}

 We therefore have the following connection between a conjugation-free group $G$ and its associated line arrangement.

\begin{prs} \label{prsCFGgroupToLineGroup}
Let $G$ be a conjugation-free group with a planar conjugation-free connected graph $Gr(G)$. Let $\LL(G)$ be the associated line arrangement. Then $\pi_1(\CC^2 - \LL(G)) \simeq G$.
\end{prs}

\begin{proof} Assume that all the edges of $Gr(G)$ are straight segments.
We know that $Gr(G) \simeq Gr_L(\LL(G))$ and thus $Gr_L(\LL(G))$ is a conjugation-free graph. Hence, $\pi_1(\CC^2 - \LL(G))$ is a conjugation-free group (by Theorem \ref{thmCFG}) such that $Gr(G) \simeq Gr(\pi_1(\CC^2 - \LL(G)))$ (by Proposition \ref{prsIsoGraphs}), where the multiplicities of the isomorphic vertices (i.e. the length of the cyclic relations) are the same.

Let $x_1,\ldots,x_n$ be the generators of $\pi_1(\CC^2 - \LL(G))$, $y_1,\ldots,y_n$  the generators of $G$. We claim that the isomorphism $f_L:\pi_1(\CC^2 - \LL(G)) \rightarrow G$ is determined by the given isomorphism on the graphs $f_G: Gr_L(\LL(G))\rightarrow Gr(G)$. Explicitly, let $Gr_L(\LL(G)) = (V_\LL, E_\LL)$ and $Gr(G) = (V_G, E_G)$, where to each vertex and edge in both graphs, the data regarding the corresponding relations and generators is also included.

If $Gr_L(\LL(G)) \simeq Gr(G)$ is empty, then $G$ is a free abelian group (see Example \ref{exmGraphCycRel}(1)), isomorphic to $\ZZ^n$, and by Definition \ref{defLineArrAss}, $\LL(G)$ is a generic line arrangement; by Remark \ref{remStrLineArr}, $\pi_1(\CC^2-\LL(G))$ is also a free abelian group of the same rank, i.e. $G \simeq \pi_1(\CC^2-\LL(G))$.

If $Gr_L(\LL(G))$ consists of one vertex $v_{1,\LL} = v_{1,\LL,(i_t,\ldots,i_1)}$, let $f_G(v_{1,\LL}) = v_{1,G} = v_{1,G,(j_t,\ldots,j_1)}$ be the isomorphic vertex. In this case we define $f_L(x_{i_k}) = y_{j_k}$ for $1 \leq k \leq t$.

Assume thus that $Gr(G)$ has at least one edge.
For every edge $e_{i,\LL} \in E_\LL$, let $f_G(e_{i,\LL}) = e_{j,G} \in E_G$ be the isomorphic edge. Thus we define $f_L(x_i) \doteq y_j$.
Consider now a vertex $v_{g,\LL} = v_{g,\LL,(i_t,\ldots,i_1)} \in V_\LL$, and let $f_G(v_{g,\LL}) = v_{h,G} = v_{h,G,(j_t,\ldots,j_1)} \in V_G$ be the isomorphic vertex.

Every vertex corresponds to a multiple relation, and is connected by several edges to other vertices, such that these edges correspond to generators participating in this multiple relation. Let $\{t_m,\ldots,t_1\} \subseteq \{1,\ldots,t\}$ (resp. $\{s_m,\ldots,s_1\} \subseteq \{1,\ldots,t\}$) be a decreasing sequence of indices such that $x_{i_{t_k}}$ corresponds to an edge $e_{i_{t_k},\LL}$ for every $1 \leq k \leq m$ (resp. $y_{j_{s_k}}$ corresponds to an edge $e_{j_{s_k},G}$). Note that by the above construction, $f_L(x_{i_{t_k}}) = y_{j_{s_k}}$ for $1 \leq k \leq m$. Let $t' = t_m$, $t'' = s_m$; therefore $f_L(x_{i_{t'}}) = y_{j_{t''}}$.

%


 Hence, we define
\begin{equation}\label{eqnDefIso}
f_L(x_{i_{t'-k}}) \doteq y_{j_{(t''-k)\text {mod}(t)}} \text{ for } 0 \leq k < t'\text{ and } f_L(x_{i_{t'+k}}) \doteq y_{j_{(t''+k)\text {mod}(t)}}
\text{ for } 1 \leq k \leq t-t'.
\end{equation}

The above defined function is compatible with the already defined function on the generators $x_{i_{t_k}}$ as otherwise the graphs $Gr_L(\LL(G)), Gr(G)$ would not be isomorphic, when considering the data attached to every vertex.

The defined function $f_L$ so far was defined on the generators (both in $G$ and in $\pi_1(\CC^2 - \LL(G))$) that participate in a multiple relation. If there are generators that do not
participate in any multiple relation, then they commute with all the other generators (this is since every two generators participate in a mutual conjugation-free cyclic relation. If this relation is not multiple, then it is of length $2$, hence a commutator).
 Assume that these generators in $\pi_1(\CC^2 - \LL(G))$ are $\{x_{w_1},\ldots,x_{w_u}\}$, and that these generators in $G$
are
$\{y_{w'_1},\ldots,y_{w'_u}\}$.
Hence define $f_L(x_{w_i}) = y_{w'_i}$ for every $1 \leq i \leq u$.

We have defined $f_L$ on every generator of $\pi_1(\CC^2 - \LL(G))$. We now have to check that $f_L$ is indeed an isomorphism, that is, we have to check that the relations are preserved. Let us note that if $n_i$ is the number of cyclic relation of length $i$ in $G$, then it is also
 the number of cyclic relation of length $i$ in $\pi_1(\CC^2 - \LL(G))$.

 Obviously, all the commutator relations are transferred  to commutator relations. Let $R_{g,\LL,(i_t,\ldots,i_1)}$ be a multiple relation
  in $\pi_1(\CC^2 - \LL(G))$, $v_{g,\LL} \in V_\LL$ be the corresponding vertex, $f_G(v_{g,\LL}) = v_{h,G} \in V_G$ the isomorphic vertex and
  $R_{h,G,(j_t,\ldots,j_1)}$  the corresponding  relation in $G$.

\begin{remark} \label{remCycRelPres}
\rm{
Assume that $H$ is a group with cyclic conjugation-free relation $R_t$ of length $t$. Let $\s = (1\,\,2 \ldots t) \in \text{Sym}_t$, $z_1,\ldots,z_t \in H$. Then it is easy to check that the cyclic relation of length $t$:
$$
R_t: z_tz_{t-1}\cdot \ldots \cdot z_1 = \cdots = z_1z_t \cdot \ldots \cdot z_2
$$
is invariant under the action of $\s$ on the indices, i.e., for every $0 \leq i < t$, if
$$
R^i_t: z_{\s^i(t)} z_{\s^i(t-1)}\cdot \ldots \cdot z_{\s^i(1)} = \cdots = z_{\s^i(1)}z_{\s^i(t)} \cdot \ldots \cdot z_{\s^i(2)},
$$
then $R_t = R^i_t$.
}
\end{remark}

We have to show that $f_L(R_{g,\LL,(i_t,\ldots,i_1)}) = R_{h,G,(j_t,\ldots,j_1)}$. Examining the isomorphism (defined at Equation (\ref{eqnDefIso})) on the generators $\{x_{i_m}\}_{m=1}^t$, we see that there exits $0 \leq r < t$ such that $f_L(x_{i_m}) = y_{j_{\s^r(m)}}$, hence
$$
f_L(R_{g,\LL,(i_t,\ldots,i_1)}) = R_{h,G,(j_{\s^r(t)},\ldots,j_{\s^r(1)})} \overset{Rem.\,\ref{remCycRelPres}}{=} R_{h,G,(j_t,\ldots,j_1)}.
$$
This proves that multiple relations are transferred to the corresponding multiple relations, and since the number of relations of any length is equal in $G$
and in  $\pi_1(\CC^2 - \LL(G))$, we are done. \end{proof}

\begin{remark} \label{remCFnoCyc}
\em{Let $G$ be a conjugation-free group such that $\beta(Gr(G))=0$. Then $\beta(Gr_L(\LL(G)))~=~0$ and  $G \simeq \pi_1(\CC^2 - \LL(G))$. Hence, by Remark \ref{remStrLineArr}, $G$ is isomorphic to a direct sum of a free abelian group and free groups.
}
\end{remark}

\subsection{The lower central series and line arrangements} \label{subsecLCSLineArr}

Let $G$ be a conjugation-free group with a  conjugation-free graph and let $\LL(G)$ be the associated line arrangement. We know that $G \simeq \pi_1(\CC^2 - \LL(G))$ by Proposition \ref{prsCFGgroupToLineGroup}. Moreover, note that if $n_i$ is the number of cyclic relation of length $i$ in $G$, then it is the number of singular points of multiplicity $i$ of $\LL(G)$.

Given any line arrangement $\LL$, denote $\phi_k(\LL) \doteq \phi_k(\pi_1(\CC^2 - \LL))$.
 Falk \cite[Proposition 3.8]{Falk2} proved that $$\phi_3(\LL) \geq \sum_{i \geq 3} n_i\omega_3(i-1),$$ where $n_i$ is the number of singular points of multiplicity $i$ (Moreover, it is proven in \cite{Falk2} that $\phi_k(\LL)$ are determined by the combinatorics of $\LL$).  Therefore, $$\phi_3(G) = \phi_3(\LL(G)) \geq \sum_{i \geq 3} n_i\omega_3(i-1).$$ By Proposition \ref{prsFirstPart}, we know that if $Gr(G)$ is a cycle-separated graph, then $$\phi_3(\LL(G)) = \phi_3(G) \leq \sum_{i \geq 3} n_i\omega_3(i-1).$$ Therefore, we can prove the following theorem:

\begin{thm}\label{thmMainPhi}
Let $G$ be a conjugation-free group whose graph is a cycle-separated graph. Then:
$$\phi_k(G) = \sum_{i \geq 3} n_i\omega_k(i-1), \emph{ for every  } k \geq 2.$$
\end{thm}

\begin{proof}
For $k=2$ the equality was proved in Section \ref{subsubsecG2G3} and
it is clear from the discussion above that the equality holds for $k=3$. Moreover, Papadima-Suciu \cite[Theorem 2.4]{PS} proves that
 given a line arrangement $\LL$ such that
 $\phi_3(\LL) = \sum_{i \geq 3} n_i\omega_3(i-1)$
  then $\phi_k(\LL) = \sum_{i \geq 3} n_i\omega_k(i-1)$ for every $k \geq 3$. \end{proof}

\begin{definition} \label{defDecomp}
\emph{Given a line arrangement $\LL$ such that $\phi_k(\LL) = \sum_{i \geq 3} n_i\omega_k(i-1)$ for every $k \geq 3$, Papadima-Suciu \cite{PS} calls  these arrangements \emph{decomposable}, since  the associated graded Lie algebra $$  \bigoplus_{i\geq 2} \pi_1(\CC^2 - \LL)_i/\pi_1(\CC^2 - \LL)_{i+1}$$ behaves as if $\pi_1(\CC^2 - \LL)$ were a direct product of free groups. Following this definition, we call a group $G$ \emph{decomposable} if $ \oplus_{i\geq 2} G_i/G_{i+1}$ behaves as if $G$ were a direct product of free groups, i.e. if $\phi_k(G) = \sum_{i \geq 3} n_i\omega_k(i-1), \text{ for every  } k \geq 2$.}
\end{definition}

\begin{remark}
\rm{A group $G$ can be decomposable even though it might be that $G$ cannot be isomorphic to any finite product of free groups of finite rank. For example, let us look at the arrangement at Figure \ref{GraphToLine}(b.IV), which is also called $X_3$.

The graph of this arrangement is a cycle of length $3$, and thus $G_{X_3} = \pi_1(\CC^2 - X_3)$ is a conjugation-free group with a cycle-separated graph, i.e. it is decomposable. However, as it is shown, for example, in \cite[Section 1]{PS}, $G_{X_3}$ cannot be isomorphic to any such product.}
\end{remark}

\begin{remark} \label{RemPhi23Top}
\rm{Given any line arrangement $\LL$ with $n$ lines, the invariants $\phi_2(\LL), \phi_3(\LL)$ can be computed via topological invariants.

(1) It is known that
$\phi_2 = a_2$, where $a_i$ is the number of minimal generators of degree $i$ in the Orlik-Solomon ideal $I$ (see \cite[Definition 3.5]{OT}), or equivalently that $a_2 = \binom{n}{2} - b_2$, where $b_2$ is the second Betti number of $\C^2 - \A$ (see e.g. \cite{Falk} or Equation (\ref{eqnB2})). Moreover, if $G$ is a conjugation-free group with a cycle-separated graph, let $\LL(G)$ be the associated line arrangement. Then by the discussion above
$\phi_2(G) = \phi_2(\LL(G))$.

(2) By \cite[Coroallary 3.6]{SS}, $\phi_3 = b_3  + b_1 (\binom{b_1}{2} - b_2)  - \binom{b_1}{3} + a_3$, where the $b_i$'s are the Betti numbers of the central $3$-arrangement, which is the cone of our line arrangement in $\C^2$. While the $b_i$'s can be read directly from the intersection lattice of $\LL$, $a_3$ cannot.
}
\end{remark}

\begin{example}

\rm{In \cite[Section 7]{PS} a series of examples of hyperplane arrangements, consisting of graphic arrangements, was given, as an example of decomposable arrangements. Let us give a different example of  decomposable arrangements which are not graphic arrangements.

Let $A_k$ be the  braid arrangement in $\C^k$; i.e. $A_k = \{\prod_{1 \leq i < j \leq k}(x_i-x_j) =0\}$. Let $A^0_k \subset \C^2$ be the intersection of $A_k$ with a generic plane of dimension $2$. In order to find a decomposable line arrangement which is not a graphic arrangement,
it is enough to find a conjugation-free line arrangement which is not a sub-arrangement of $A^0_k$ (as every graphic arrangement is a sub-arrangement of $A_k$). The graph $Gr_L(A^0_k)$ has only vertices associated to multiple points of multiplicity $3$. This is since the intersection lattice of $A_k$ is the partition lattice (see e.g. \cite[Proposition 2.9]{OT}), and hence the codimension-two intersections have an intersection lattice equal to the rank-three truncation of the partition lattice. As can be easily checked, this truncation shows that the lines of $A^0_k$ intersect only at nodes or singular points with multiplicity $3$.

Therefore, a real line arrangement $\LL$, whose graph $Gr_L(\LL)$ is cycle-separated with some vertices associated to singular points whose multiplicity is greater or equal to $4$, cannot be a graphic arrangement, though it is decomposable (as  $Gr_L(\LL)$ is cycle-separated).


Note also that the arrangement $X_2$ (see Figure \ref{X2}(a) and e.g. \cite[Example 10.4]{Suc}) is decomposable but does not have a conjugation-free graph (see Figure \ref{X2}(b)).
 \begin{figure}[h!]
\epsfysize 4cm
\epsfbox{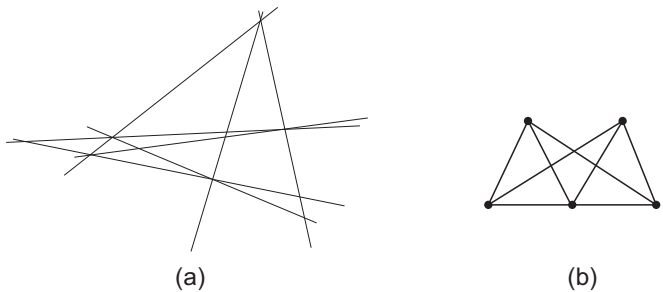}
\caption{Part (a): The arrangement $X_2$. Part (b): The graph $Gr_L(X_2)$ of the arrangement.}
\label{X2}
\end{figure}
}

\end{example}

\begin{remark}
\rm{The LCS formula of Falk and Randell \cite{FaRa} expresses the
ranks $\phi_k(\LL)$ of a fiber-type arrangement $\LL$ in terms of the Poincar\'{e} polynomial $P(M,t) = \sum b_it^i$ of the complement $M = \C^2 - \LL$:
$$
\prod_{k=1}^{\infty} (1-t^k)^{\phi_k(\LL)} = P(M,-t).
$$
Note that there are arrangements for which the LCS formula holds but the arrangements are not decomposable (e.g. the deleted $B_3$-arrangement, see \cite[Example 10.6]{Suc}), and there are arrangements for which the LCS formula does not hold but the arrangements are decomposable (e.g. the $X_3$ arrangement). Other LCS formulas were also obtained, for other types of arrangements; for example, for graphic arrangements \cite{LFS}.

Moreover, for decomposable line arrangements $\LL$ with $n$ lines there is another LCS formula, found by Papadima-Suciu \cite[Corollary 2.6]{PS}:
$$
\prod_{k=1}^{\infty} (1-t^k)^{\phi_k(\LL)} = (1-t)^{n-b_2(\LL)}\cdot \prod_{p \in Sing(\LL)} (1-(m(p)-1)t),
$$
where $m(p)$ is the multiplicity of the singular point $p$ and

\begin{equation} \label{eqnB2}
b_2(\LL) = \sum_{p \in Sing(\LL)} (m(p)-1). \end{equation}

 Note that this equation can also be stated for a conjugation-free group $G$ with a cycle-separated graph, generated by $n$ generators:
$$
\prod_{k=1}^{\infty} (1-t^k)^{\phi_k(G)} = (1-t)^{n-b_2(G)}\cdot \prod_{R \in Rel(G)} (1-(len(R)-1)t),
$$
where $Rel(G)$ is the set of relations of $G$, $len(R)$ is the length of the cyclic relation $R$ and \linebreak $b_2(G)=\sum_{R \in Rel(G)} (len(R)-1)$.
}
\end{remark}

Based on computations made for other arrangements, whose graph is not cycle-separated, we conclude this paper with the
following conjecture:
\begin{conjecture} \label{conjCF}
Let $\LL$ (resp. $G$) be  a conjugation-free line arrangement (resp. group). Then $\LL$ (resp. $G$) is decomposable.
\end{conjecture}

\end{document}